\documentclass[11pt,a4paper]{scrartcl}
\usepackage[utf8]{inputenc}
\usepackage[T1]{fontenc}
\usepackage[british]{babel}
\usepackage{csquotes}
\usepackage{amsmath,mathtools}
\usepackage{amsfonts}
\usepackage{amssymb}
\usepackage{makeidx}
\usepackage{graphicx}
\usepackage{hyperref}
\usepackage{color}
\usepackage{amsthm}
\usepackage{mathrsfs}
\usepackage{float}
\usepackage{enumerate}
\usepackage{comment}
\usepackage{caption}
\usepackage{subcaption}

\hypersetup{
	colorlinks=true,
	linkcolor=blue,
	citecolor=blue,
	filecolor=blue,      
	urlcolor=blue
}

\usepackage{tikz, pgf}

\title{
Convergence of non-reversible Markov processes via lifting and flow Poincaré inequality}
\author{
Andreas Eberle\thanks{Institute for Applied Mathematics, University of Bonn. E-Mail: \href{mailto:eberle@uni-bonn.de}{eberle@uni-bonn.de}}\qquad 
Arnaud Guillin\thanks{Laboratoire de Math\'ematiques Blaise Pascal, CNRS and Universit\'e Clermont Auvergne. E-Mail: \href{mailto:arnaud.guillin@uca.fr}{arnaud.guillin@uca.fr}}\qquad 
Leo Hahn\thanks{Université de Neuchâtel. E-Mail: \href{mailto:}{leo.hahn@unine.ch}}\\ 
Francis L\"orler\thanks{Institute for Applied Mathematics, University of Bonn. E-Mail: \href{mailto:loerler@uni-bonn.de}{loerler@uni-bonn.de}}\qquad
Manon Michel\thanks{Laboratoire de Math\'ematiques Blaise Pascal, CNRS and Universit\'e Clermont Auvergne. E-Mail: \href{mailto:manon.michel@uca.fr}{manon.michel@uca.fr}}}
\date{\today}

\newcommand{\Lhat}{{\hat{\mathcal L}}}
\newcommand{\mL}{{\mathcal L}}
\newcommand{\E}{\mathcal{E}}
\newcommand{\Ld}{{\hat{\mathcal L}_\mathrm{tr}}}
\newcommand{\Lv}{{\hat{\mathcal L}_v}}

\newcommand{\myAngle}[1]{\left< #1 \right>_{L^2(\lambda \otimes \mu)}}

\newcommand{\myVert}[1]{\left\Vert #1 \right\Vert_{L^2(\lambda \otimes \mu)}}
\newcommand{\myVertHat}[1]{\left\Vert #1 \right\Vert_{L^2(\lambda \otimes \hat\mu)}}

\newcommand{\LvPoincareConstant}{{\frac1{m_v}}}
\newcommand{\LvInvPoincareConstant}{m_v}

\newcommand{\Shat}  {{\hat{\mathcal S}}}
\renewcommand{\S}   {\mathcal{S}}
\newcommand{\V}     {\mathcal{V}}
\newcommand{\diff}  {\mathop{}\!\mathrm{d}}
\newcommand{\comp}  {\mathrm{c}}
\newcommand{\R}     {\mathbb{R}}
\newcommand{\D}     {\mathrm{D}}

\DeclareMathOperator{\dom}  {Dom}
\DeclareMathOperator{\Gap}  {Gap}
\let\Re\relax

\DeclareMathOperator{\Re}   {Re}

\DeclareMathOperator{\Var}  {Var}

\DeclareMathOperator{\spn}  {span}

\theoremstyle{plain}
\newtheorem{Thm}{Theorem}
\newtheorem{Lem}[Thm]{Lemma}

\newtheorem{corollary}[Thm]{Corollary}

\theoremstyle{definition}
\newtheorem{Def}[Thm]{Definition}
\newtheorem{Rem}[Thm]{Remark}
\newtheorem{Exa}[Thm]{Example}

\newtheorem{Asm}{Assumption}
\newtheorem{asm}{Assumption}

\begin{document}

\maketitle

\begin{abstract}
    We propose a general approach for quantitative convergence analysis of non-reversible Markov processes, based on the concept of second-order lifts and a variational approach to hypocoercivity. To this end, we introduce the flow Poincaré inequality, a space-time Poincaré inequality along trajectories of the semigroup, and a general divergence lemma based only on the Dirichlet form of an underlying reversible diffusion.
    We demonstrate the versatility of our approach by applying it to a pair of run-and-tumble particles with jamming, a model from non-equilibrium statistical mechanics, and several piecewise deterministic Markov processes used in sampling applications. Our framework in particular includes processes with general stochastic jump kernels.

    \par\vspace\baselineskip
    \par
\end{abstract}

\section{Introduction}

For reversible Markov processes, 
the inverse Poincaré constant or spectral gap is a powerful tool for deriving quantitative convergence bounds, as it coincides with the rate of convergence and the asymptotic decorrelation time.  For non-reversible processes, however, the inverse Poincaré constant can be much lower than the convergence rate, and the process may not even satisfy a Poincaré inequality but still converge rapidly. In fact, acceleration of convergence stemming from non-reversibility has been shown in different models \cite{HHS93,diaconis00,chen99,HHS05,LNP13,GM16,CH13,EGZ19} and
observed in many more \cite{Horowitz91,MKK14,NMKH15,bierkens19,Neal2011HMC}. The design of fast non-reversible processes often relies on the lifting procedure, which introduces an enlarged position-velocity state space to incorporate persistence, as introduced in \cite{diaconis00} for discrete chains. In the continuous-time and -space setting, this approach has led to the use of piecewise deterministic Markov processes (PDMPs) \cite{davis84}. These processes combine deterministic motion with discrete velocity jumps at event times, along with velocity refreshments typically occurring after exponentially distributed waiting times \cite{MKK14,bierkens19,bouchardcote18}.  Despite their impressive numerical performance, the analytical understanding of these processes remains limited, posing challenges for rigorous theoretical analysis. Therefore, obtaining quantitative bounds on the rate of convergence of non-reversible processes has attracted a lot of attention, while being much more involved than in the reversible case. For degenerate diffusions, this active area of research has become known as hypocoercivity \cite{herau07,villani09,DMS15}. Recently, a variational approach to hypocoercivity pioneered by Albritton, Armstrong, Mourrat and Novack \cite{albritton2024variational} was applied by Cao, Lu and Wang \cite{cao23} for Langevin dynamics, for the first time obtaining rates of convergence of the correct order. This approach has further been applied to discretisations of Langevin dynamics \cite{Lehec25} and to show that some PDMPs exhibit sharp dimensional dependence \cite{lu22}.

Inspired by the notion of lifts of discrete Markov chains \cite{diaconis00,chen99}, second-order lifts of reversible diffusions \cite{eberle24} were introduced as their counterpart in continuous time and space. This enables the analysis of the convergence rate of non-reversible Markov processes by yielding an upper bound in terms of the spectral gap of an associated reversible collapsed process, generalising the related result for discrete Markov chains \cite{chen99} by showing that the relaxation time can at most be improved by a square root through lifting. Furthermore, the variational approach to hypocoercivity \cite{albritton2024variational,cao23,lu22} can be simplified by phrasing it in the language of second-order lifts, allowing for a systematic derivation of lower bounds on the convergence rates. The second-order lift property enables the exploitation of properties of the simpler collapsed process to obtain bounds on the convergence rate of the lifted process. 

\medskip

The principle of energy dissipation, i.e.\ differentiating the $L^2$-norm along the semigroup and controlling this dissipation by the original $L^2$-norm, is often used to get an exponential decay in $L^2$. In the elliptic setting, where $(P_t)_{t\geq 0}$ is the transition semigroup of a reversible Markov process with invariant probability measure $\mu$ and generator $\mL$, this can be achieved through the Poincaré inequality
\begin{equation*}
    \lVert f\rVert_{L^2(\mu)}^2\ \leq\ \frac{1}\nu\bigl\langle f,-\mL f\bigr\rangle_{L^2(\mu)}
\end{equation*}
for all $f\in L^2(\mu)$ with mean zero, as it implies that $\lVert P_tf\rVert_{L^2(\mu)}\leq e^{-\nu t}\lVert f\rVert_{L^2(\mu)}$ for all $t\geq 0$. 
However, such a Poincaré inequality is no longer valid in the hypoelliptic, non-reversible setting. Instead, to obtain exponential decay of the transition semigroup of a possibly non-reversible Markov process, we introduce the \emph{flow Poincar\'e inequality} 
\begin{equation*}
    \int_0^T\lVert P_tf\rVert^2_{L^2(\mu)}\diff t\ \leq \ \frac1\nu\int_0^T\bigl\langle P_tf,-\mL P_t f\bigr\rangle_{L^2(\mu)}\diff t
\end{equation*}
for all $f\in L^2(\mu)$ with mean zero, where $T>0$ is a fixed time length. Such a flow Poincaré inequality amounts to a time-integrated Poincaré inequality along trajectories of the semigroup, and is equivalent to dissipation of time-averaged $L^2$-norms, as considered in \cite{albritton2024variational,cao23,eberle24}. In particular, it implies the existence of a constant $C>0$ such that
\begin{equation*}
    \lVert P_t f\rVert_{L^2(\mu)}\ \leq\ Ce^{-\nu t}\lVert f\rVert_{L^2(\mu)}\qquad\text{for all }t\geq 0
\end{equation*}
and all mean-zero square-integrable observables $f$. In \cite{DMS15,villani09}, a modified $L^2$-norm based on an additional mixing term is introduced to ensure coercivity and hence control of the dissipation. The strength of an approach based on the flow Poincaré inequality lies in the more intrinsic nature of the time-averaged $L^2$-norm as opposed to adding an artificial mixing term to the norm, allowing for sharper rates.
By exploiting the second-order lift property implicitly considered in \cite{albritton2024variational,lu22,cao23}, we prove such a flow Poincaré inequality in Theorem~\ref{thm:FPIlift}. The proof is based on a divergence lemma that only relies on the Dirichlet form of the underlying collapsed process, see Lemma~\ref{lem:divergence_lemma}. 
Previously, such divergence lemmas were only considered for the overdamped Langevin diffusion. Our result is more general and enables the treatment of non-reversible processes whose collapse is not an overdamped Langevin process. This is the case when incorporating non-trivial boundary conditions, as is often the case in non-equilibrium statistical mechanics.
Finally, our framework allows for a simple treatment of PDMPs with more general refreshment and jump mechanisms.

We demonstrate the versatility of our approach by applying it to two first examples. Firstly, we consider a pair of jamming run-and-tumble particles (RTPs), a model of active particles which maintain persistent movement, lying at the interface of out-of-equilibrium statistical mechanics and biology. Their speed of convergence to equilibrium is physically relevant, as it is linked to the time scale of the onset of out-of-equilibrium phenomena~\cite{vicsek12,elgeti15}, such as motility-induced phase separation~\cite{cates15}. The prevailing method in the statistical physics literature to address this question is the spectral approach~\cite{mallmin19,angelani19,malakar18}, leading to convergence speeds which are only valid asymptotically due to non-reversibility. Hence the value of the non-asymptotic total variation mixing time results~\cite{GHM24,hahn25}, extended to the $L^2$ relaxation time framework in the present article. Secondly, we consider piecewise deterministic Markov processes (PDMPs) used in sampling applications~\cite{MDS20,bierkens19}. Due to the degeneracy introduced through their definition on an enlarged position-velocity state space, the analysis of their convergence is delicate~\cite{BRZ19,durmus20,bierkens22,lu22}.

\medskip
The paper is divided into four more sections. In Section~\ref{sec:lift_collapse}, we begin by summarising and illustrating the concept of second-order lifts using overdamped Langevin diffusions as a simple example. We then focus on the more intricate case of a pair of jamming RTPs, demonstrating how this system can be understood as a lift of sticky Brownian motion. Then Section~\ref{sec:main_results} presents the main results on convergence to equilibrium of lifts via the flow Poincaré inequality.
Section~\ref{sec:rtp} applies our framework to a pair of RTPs with jamming.
Finally, Section~\ref{sec:pdmps} presents applications to PDMPs for sampling, namely the Forward process and variants of the Zig-Zag process.

\section{Second-order lifts}\label{sec:lift_collapse}

We begin by summarising the definition of second-order lifts and the related upper bound on the convergence rate and then illustrate the concept through lifts of overdamped Langevin diffusions. These include Langevin dynamics, randomised Hamiltonian Monte Carlo and many PDMPs introduced in sampling applications.  We subsequently introduce the system formed by a pair of jamming RTPs on the torus and show how it can be understood as a lift of a sticky Brownian motion. Such a system presents unique challenges within the second-order lift framework. Specifically, the presence of boundaries at which the process spends a non-zero amount of time complicates the identification of the underlying collapsed process. Additionally, the RTP velocity dynamics differ from previously considered mechanisms, requiring distinct proof techniques. The convergence of such systems will be explored in greater detail in Section~\ref{sec:rtp}.

\subsection{Definition and examples}

Consider a time-homogenous Markov process $(\hat Z_t)$ with invariant probability measure $\hat\mu$ on a Polish space $\Shat$. Let $\pi\colon\Shat\to\S$ be a measurable surjection onto another Polish space, and let $(Z_t)$ be a reversible diffusion process with state space $\S$ whose invariant probability measure is the pushforward $\mu = \hat\mu\circ\pi^{-1}$ of $\hat\mu$ under $\pi$. We denote their associated transition semigroups on $L^2(\hat\mu)$ and $L^2(\mu)$ by $(\hat P_t)$ and $(P_t)$, and the associated generators by $(\Lhat,\dom(\Lhat))$ and $(\mL,\dom(\mL))$, respectively. The Dirichlet form associated to $(Z_t)$ is the extension of 
\begin{equation*}
    \E(f,g)=-\int_\S f\mL g\diff\mu
\end{equation*}
to a closed symmetric bilinear form with domain $\dom(\E)$ given by the closure of $\dom(\mL)$ with respect to the norm $\|f\|_{L^2(\mu)} + \E(f,f)^{1/2}$. 

\begin{Def}[Second-order lifts \cite{eberle24}]\label{def:lift}
The process $(\hat Z_t)$ is a \emph{second-order lift} of $(Z_t)$ if there exists a core $C$ of $(\mL,\dom(\mL))$ such that
\begin{align}\label{eq:deflift0}
    f \circ \pi \in \dom(\Lhat) \text{ for all } f \in C
\end{align}
and for all $f,g\in C$ we have
\begin{equation}\label{eq:deflift1}
     \big< \Lhat (f \circ \pi), g \circ \pi \big>_{L^2(\hat\mu)} = 0,
\end{equation}
and
\begin{equation}\label{eq:deflift2}
    \frac{1}{2}\big< \Lhat (f \circ \pi), \Lhat(g \circ \pi) \big>_{L^2(\hat\mu)} = -\left< f, \mL g\right>_{L^2(\mu)} = \mathcal E(f, g).
\end{equation}
If in addition
\begin{align*}
    \{ f \in L^2(\mu)\colon f \circ \pi \in \text{\normalfont Dom}(\Lhat)\} = \dom(\E),
\end{align*}
then $(\hat Z_t)$ is called a \emph{strong second-order lift} of $(Z_t)$. Conversely, we refer to the process $(Z_t)$ as the \emph{collapse} of $(\hat Z_t)$.
\end{Def}

We also say that $(\Lhat,\dom(\Lhat))$ or $(\hat P_t)$ is a second-order lift of $(\mL,\dom(\mL))$ or $(P_t)$, respectively. In the rest of this paper, we simply refer to second-order lifts as lifts. Usually, $\Shat=\S\times\V$ will be a product space, $\pi(x,v)=x$ is the projection on $\S$, and we think of $(x,v)\in\Shat$ as a position and a velocity. More generally, $\S$ could for example be a Riemannian manifold and $\Shat$ the tangent bundle, see \cite{eberle24b}.

To get a better understanding of the definition of lifts, disintegrate the measure $\hat\mu$ as
\begin{equation*}
    \hat\mu(A) = \int_\S\kappa_x(A)\mu(\diff x)\qquad\text{for all measurable }A\subseteq\Shat,
\end{equation*}
i.e.\ $\hat\mu(\diff x\diff v) = \mu(\diff x)\kappa_x(\diff v)$ in case $\Shat=\S\times\V$.
Then, under minor technical assumptions, see \cite[Remark 2 (v)]{eberle24}, the first- and second-order conditions \eqref{eq:deflift1} and \eqref{eq:deflift2} are equivalent to the limit
\begin{equation*}
    \int_{\V} \hat P_t(g\circ\pi)(x,v)\kappa_x(\diff v)=(P_{t^2}g)(x)+o(t^2)
\end{equation*}
in $L^2(\mu)$ as $t$ goes to 0. Hence, for small values of $t$, the transition semigroup $(P_{t^2})$ of the collapse behaves approximately as an averaged version of the transition semigroup $(\hat P_t)$. Compared to lifts of Markov chains as introduced in \cite{chen99,diaconis00}, the interplay between the quadratic and linear timescales can be directly obtained in the continuous-time case.

By the spectral theorem, the rate of convergence to equilibrium in $L^2(\mu)$ of the reversible process $(Z_t)$ is accurately measured by the inverse of the spectral gap
\begin{equation*}
    \Gap(\mL) = \inf\bigl\{\Re(\alpha)\colon\alpha\in\mathrm{spec}(-\mL|_{L_0^2(\mu)})\bigr\}
\end{equation*}
of its generator $\mL$, where $L_0^2(\mu) = \{f\in L^2(\mu)\colon\int f\diff\mu=0\}$ denotes the subspace of mean-zero functions. This is no longer true for non-reversible Markov processes since their generators are not self-adjoint. Even in the exponentially ergodic case, the associated transition semigroup only satisfies 
\begin{equation}\label{eq:expdecay}
    \|\hat P_tf\|_{L^2(\hat\mu)}\le Ce^{-\nu t}\|f\|_{L^2(\hat\mu)}\qquad\text{for all }f\in L_0^2(\hat\mu)
\end{equation}
with $\nu\in[0,\infty)$ and some constant $C\in[1,\infty)$.
However, the rate of convergence to equilibrium of a lift can be bounded above by the square root of that of the collapsed process in the following way.
\begin{Thm}[\text{\cite[Theorem 11]{eberle24}}]\label{thm:upperbound}
    Suppose that $(\hat P_t)$ is a second-order lift of $(P_t)$. Assume that there exist constants $C\in[1,\infty)$ and $\nu\in[0,\infty)$ such that \eqref{eq:expdecay} is satisfied.
    Then 
    \begin{equation}\label{eq:rate}
        \nu\le (1+\log C){\sqrt{2\,\mathrm{Gap}(\mathcal L)}}.
    \end{equation}
\end{Thm}
Theorem~\ref{thm:upperbound} provides a bound on the rate of convergence achievable through lifting in terms of the spectral gap of the collapsed process, generalising a related result \cite[Theorem 3.1]{chen99} in discrete time and space. This motivates the search for optimal lifts, in the sense that they satisfy \eqref{eq:expdecay} with $\nu$ of order $\sqrt{\Gap(\mL)}$ when $\Gap(\mL)$ is small. Proving such a statement requires quantitative lower bounds on the rate of convergence of non-reversible Markov processes. We show in Section~\ref{sec:main_results} how to obtain such bounds using the framework of lifts under general assumptions on the collapsed process and the structure of the lift. The result is demonstrated in several applications in Sections~\ref{sec:rtp} and \ref{sec:pdmps}.

\begin{Exa}[Lifts of overdamped Langevin diffusions]\label{ex:langevinlift}
    Prototypical examples of second-order lifts are given by lifts of overdamped Langevin diffusions. Let $\S=\R^d$ and consider a potential $U \in C^1(\mathbb R^d)$ such that $U(x) \rightarrow +\infty$ when $|x| \rightarrow +\infty$ and $\int e^{-U(x)} \diff x < +\infty$. 

    The \emph{overdamped Langevin diffusion} is the process $(Z_t)$ with state space $\S$ given by the stochastic differential equation 
    \begin{align*}
        \diff Z_t = -\frac12 \nabla U(Z_t) \diff t + \diff B_t,
    \end{align*}
    where $(B_t)_{t\geq 0}$ is a standard $d$-dimensional Brownian motion. Its unique invariant probability measure is the Boltzmann-Gibbs measure
    \begin{equation*}
        \mu(\diff x)\propto e^{-U(x)}\diff x,
    \end{equation*}
    and the associated generator in $L^2(\mu)$ is
    \begin{equation*}
        \mL f = -\frac{1}{2}\nabla U\cdot\nabla f + \frac12\Delta f = -\frac{1}{2}\nabla^*\nabla f
    \end{equation*}
    for all $f\in C_\comp^\infty(\R^d)$, which is a core of the generator \cite[Theorem 3.1]{Wielens1985Selfadjointness}. Here $\nabla^*$ is the adjoint of $\nabla$ in $L^2(\mu)$.
    
    Lifts of overdamped Langevin diffusions are now given by Markov processes on $\Shat=\S\times\V = \R^d\times\R^d$, see \cite[Example 3]{eberle24}. They include processes based on Hamiltonian dynamics, such as the Langevin diffusion or randomised Hamiltonian Monte Carlo (RHMC), but also other piecewise deterministic Markov processes (PDMPs) used in sampling applications, that we will treat in Section~\ref{sec:pdmps}. These processes have in common that their invariant probability measure is a product measure $\hat\mu = \mu\otimes\kappa$, where $\kappa$ is for example the $d$-dimensional standard normal distribution, the uniform distribution on the hypercube $\{-1,+1\}^d$, or the uniform distribution on the coordinate directions $\{\pm e_i\colon i=1,\dots,n\}$. For functions $f\in\dom(\mL)$, their generators take the form $\Lhat (f\circ\pi)(x,v) = v\cdot\nabla f(x)$, so that \eqref{eq:deflift1} and \eqref{eq:deflift2} can immediately be verified.
\end{Exa}

\subsection{Jamming run-and-tumble particles as a lift of sticky Brownian motion}\label{sec:rtp_introduction}

Run-and-tumble particles (RTPs), a close approximation for the movement of bacteria such as E.~coli~\cite{schnitzer93,berg04}, alternate between long straight runs and rapid reorientations. In addition to their use in biological modelling, they are a topic of investigation in statistical mechanics as they break time reversibility at the particle level and achieve persistent motion. This can drive them out of thermodynamic equilibrium and lead to interesting phenomena not found in passive systems, such as motility-induced phase separation~\cite{cates15} where the particles form clusters. Even a single RTP displays rich long-term behaviour not observed in passive systems, such as accumulation at boundaries~\cite{angelani17,malakar18,dhar19}. From a mathematical perspective, systems of RTPs can be framed as piecewise deterministic Markov processes~\cite{hahn23,GHM24}, for which many fundamental questions, such as the regularity of invariant measures, remain open.

\begin{figure}[H]
    \centering
	\includegraphics[width=5cm]{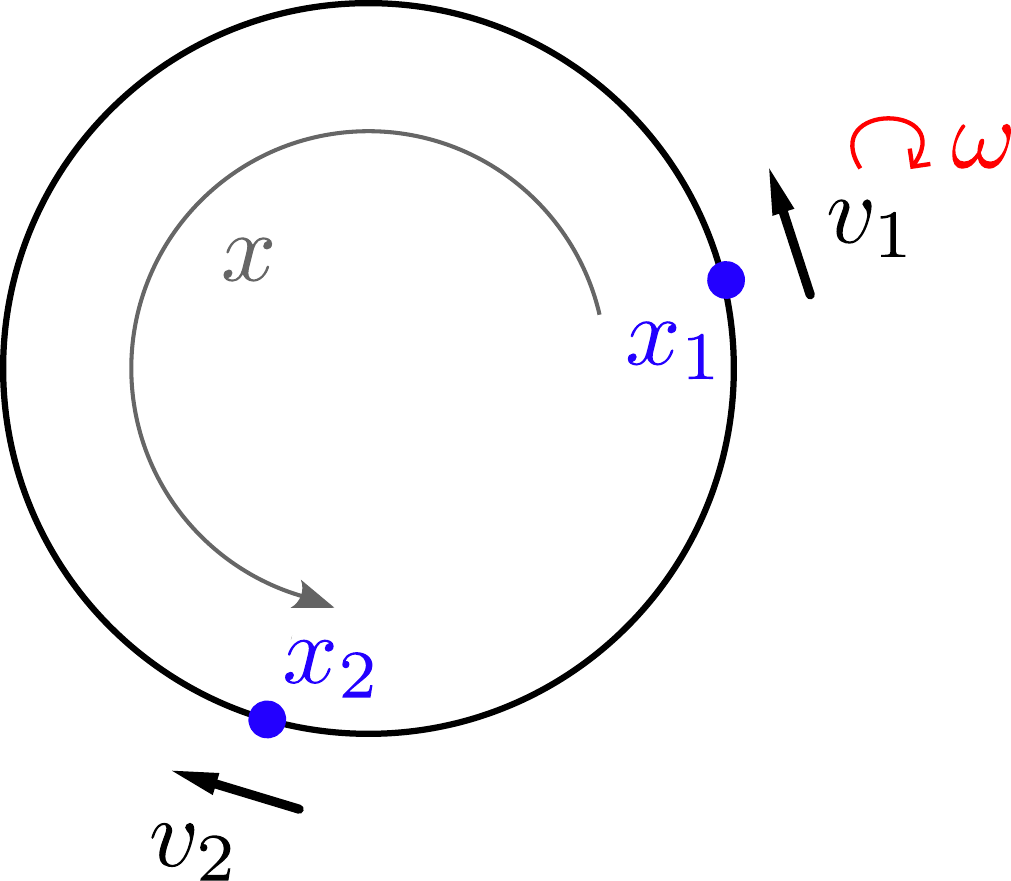}
    \caption{Two run-and-tumble particles on the 1D torus}
    \label{fig:two_RTPs_on_1D_torus}
\end{figure}

 Here, we focus on the two-particle case and start with an informal description of the model. Consider two point particles on the 1D torus of length $L$ described by their positions $x_1, x_2 \in \mathbb R / L\mathbb Z$ and their velocities $v_1, v_2 \in \{\pm 1\}$ (see Figure~\ref{fig:two_RTPs_on_1D_torus}). The velocities are independent Markov jump processes switching between $+1$ and $-1$ with rate $\omega > 0$, corresponding to the \emph{instantaneous} tumble mechanism. Between two velocity switches, each particle $x_i$ moves around the 1D torus with constant velocity $v_i$ until the two particles collide. Colliding particles jam and stop moving until a velocity flip allows them to separate.

A discrete version of this process was introduced in~\cite{slowman16} where jammed configurations were found to have increased mass under the invariant measure. This is a form of clustering hinting at motility-induced phase separation. A continuous-space version was later considered in~\cite{hahn23} for velocities following an arbitrary Markov jump process, revealing two universality classes for the invariant measure.

\begin{figure}[H]
    \centering
    \includegraphics[width=5cm]{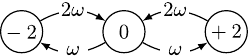}
    \caption{Jump rates of the relative velocity \( v = v_2 - v_1 \)}
    \label{fig:relative_velocity_rates}
\end{figure}

When modelling the two-particle system as a Markov process, it suffices to keep track of the separation from the first to the second particle $x \in \S =  [0, L]$ (see Figure~\ref{fig:two_RTPs_on_1D_torus}) and the relative velocity $v = v_2 - v_1 \in \V = \{-2, 0, 2\}$ (see Figure~\ref{fig:relative_velocity_rates}). Figure~\ref{fig:RTP_process_realization} shows a realisation of the process.

\begin{Def}[Jamming RTP process \cite{GHM24}] Let \( (x_0, v_0) \in [0, L] \times \{-2, 0, +2\} \) and $\omega\in(0,\infty)$ be given. Let the relative velocity $v(t)$ be a Markov jump process with initial state $v_0$ and the jump rates of Figure~\ref{fig:relative_velocity_rates}. Denote $0 = T_0 < T_1 < \cdots$ its jump times, i.e.\ $v(t)=v(T_n)$ for $t\in[T_n,T_{n+1})$. Recursively define the particle separation $x(t)$ by $x(0) = x_0$ and
$$
    x(t) = \max\left[0, \min\left[L, x(T_n) + v(T_n)(t - T_n)\right]\right]\quad\text{for } t \in [T_n, T_{n+1}).
$$
We call $X(t) = (x(t), v(t))$ the jamming RTP process on $[0,L]$ with parameter $\omega$. 
\end{Def}

\begin{figure}[H]
    \centering
    \includegraphics[width=10cm]{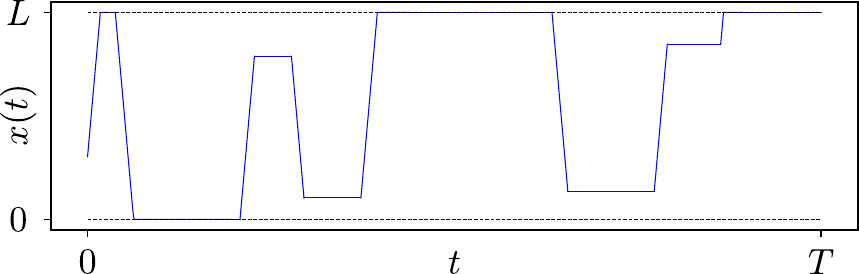}
    \caption{Realisation of the jamming RTP process}
    \label{fig:RTP_process_realization}
\end{figure}

Henceforth,  for conciseness, and when no confusion arises, we will occasionally refer to the jamming RTP process simply as the RTP process. The following theorem gives the invariant measure and the generator of the jamming RTP process. Its proof is relegated to Section~\ref{ssec:RTPdomainproof}.

\begin{Thm}[Invariant measure and generator of the jamming RTP process]\label{thm:invariant_measure_generator_RTP} 

Let $\omega,L\in(0,\infty)$.
\begin{enumerate}[(i)]
    \item (\cite{hahn23}) The jamming RTP process on $[0,L]$ with parameter $\omega$ takes its values in $\Shat = \S \times \V = [0, L] \times \{ -2, 0, +2\}$ and its unique invariant probability is given by
        $$
        \hat \mu = \sum_{v \in \mathcal V} \left( d^0_v \delta_0 + a_v \diff x + d^{L}_v \delta_{L} \right) \otimes \delta_v ,\qquad\text{where}
    $$
    \begin{align*}
    d^0_2 &= 0, & a_2 &= \frac{\omega/4}{2 + \omega {L}}, & d^{L}_2& = \frac{1/2}{2 + \omega {L}}, \\
d^0_0 &= \frac{1/2}{2 + \omega {L}}, & a_0 &= \frac{\omega/2}{2 + \omega {L}}, & d^{L}_0 &= \frac{1/2}{2 + \omega {L}}, \\
d^0_{-2} &= \frac{1/2}{2 + \omega {L}}, & a_{-2} &= \frac{\omega/4}{2 + \omega {L}}, & d^{L}_{-2} &= 0.
    \end{align*}
    \item The domain $\dom(\Lhat)$ of its $L^2(\hat\mu)$-generator $\Lhat$ is the set of functions $f \in L^2(\hat\mu)$ s.t.
    $$
        f(x, \pm 2) \in H^1(0, L), \qquad \lim_{x\to L} f(x, 2) = f(L, 2), \qquad \lim_{x\to 0} f(x, -2) = f(0, -2).
    $$
    For $f \in \dom(\Lhat)$ one has
    \begin{equation}\label{eq:RTPgenerator}
        \Lhat f(x, v) = v 1_{\{ 0 < x < L \}} \partial_x f(x, v) + \sum_{\tilde v \in \V} \lambda_{v\tilde v} f(x, \tilde v),
    \end{equation}
    where
	\[
		\begin{pmatrix}
			\lambda_{+2, +2} & \lambda_{+2, 0} & \lambda_{+2, -2} \\
			\lambda_{0, +2} & \lambda_{0, 0} & \lambda_{0, -2} \\
			\lambda_{-2, +2} & \lambda_{-2, 0} & \lambda_{-2, -2} \\
		\end{pmatrix}
		=
		\begin{pmatrix}
		-2\omega & 2\omega & 0 \\
		\omega & -2\omega & \omega \\
		0 & 2\omega & -2\omega
		\end{pmatrix}
	\]
	is the discrete generator of the relative velocity (see Figure~\ref{fig:relative_velocity_rates}).
\end{enumerate}
\end{Thm}

Regarding the second-order lift framework, this process differs from previously studied PDMPs and is interesting for two reasons. Firstly, the state space has boundaries and the process spends a non-zero amount of time on them. This is particularly interesting as the behaviour at boundaries usually manifests itself not in the formula for the generator but rather in its domain. Because second-order lifts are defined in terms of generators, even identifying the associated collapsed process becomes delicate. Secondly, different velocity processes necessitate different proofs, as was previously done when the velocity is an Ornstein–Uhlenbeck process \cite{cao23,eberle24b} and when the velocity is completely resampled at Poisson times \cite{lu22,eberle24}. The jamming RTP process's velocity, however, is a Markov jump process with the rates of Figure~\ref{fig:relative_velocity_rates}, requiring a new approach. 

In \cite{hahn23}, this particular jamming RTP process is shown to belong to a class of near-equilibrium systems, whose steady states differ from those of equilibrium diffusive particles only by the presence of positive mass on the boundaries. This suggests that the jamming RTP process can be understood as the lift of a sticky diffusion, whose invariant measure is the image measure
$$
    \mu = \hat\mu\circ\pi^{-1} = \frac1{2+\omega L} \delta_0 + \frac\omega{2+\omega L} \diff x + \frac1{2+\omega L} \delta_L
$$
under the projection $\pi(x, v) = x$.
An integration by parts argument further supports this interpretation by pointing to a diffusion, as
$$
    \int \Lhat(f \circ \pi) \Lhat(g \circ \pi) \diff\hat\mu = -\frac{2\omega}{2+\omega L} \int_0^L f'' g \diff x
$$
for all $f, g \in C^\infty([0, L])$ with $\mathrm{supp}(f), \mathrm{supp}(g) \subset (0, L)$.  These heuristic arguments are made rigorous in Theorem~\ref{thm:RTP_collapse}.

\begin{Def}[Sticky Brownian motion]\label{def:stickybm}
For $\omega\in(0,\infty)$, sticky Brownian motion on $[0,L]$ with parameter $\omega$ is the $[0, L]$-valued Feller process with generator
$$
    \mL_{C^0} f = f'' ,\quad
    \dom(\mL_{C^0}) = \{ f \in C^2([0, L])\colon f''(0) = \omega f'(0) \text{ and } f''(L) = -\omega f'(L)\}.
$$
\end{Def}

Conventionally, sticky Brownian motion is defined with $\mL_{C^0}f = \frac{1}{2}f''$, see \cite[Section 3.5.2]{liggett10}. The process defined in Definition~\ref{def:stickybm} would be a sticky Brownian motion accelerated by a factor $2$. However, leaving out the factor $\frac{1}{2}$ is more convenient for our purpose. In this setting, the tumble rate $\omega$ quantifies the \emph{stickiness} at the boundaries. The limits $\omega\to0$ and $\omega\to+\infty$ describe the asymptotic \emph{persistent} and \emph{diffusive} dynamical regimes \cite{hahn23,GHM24} and correspond here to absorption and reflection at the boundaries, respectively.

\begin{Thm}[Jamming RTP is a lift of sticky Brownian motion]\label{thm:RTP_collapse}Let $\omega,L\in(0,\infty)$.
\begin{enumerate}[(i)]
    \item Sticky Brownian motion on $[0,L]$ with parameter $\omega$ has a unique invariant probability measure given by
    $$
        \mu = \frac1{2+\omega L} \delta_0 + \frac\omega{2+\omega L} \diff x + \frac1{2+\omega L} \delta_L.
    $$
    \item Its $C^0$-generator $(\mL_{C^0}, \dom(\mL_{C^0}))$ is a core of its $L^2(\mu)$-generator $(\mL, \dom(\mL))$. 
    \item The jamming RTP process on $[0,L]$ with parameter $\omega$ is a lift of sticky Brownian motion on $[0,L]$ with parameter $\omega$.
\end{enumerate}
\end{Thm}

\begin{proof} (i) Follows from the generator characterisation of invariance~\cite[Theorem 3.37]{liggett10}.

(ii) The set $\dom(\mL_{C^0})$ is dense in $L^2(\mu)$ and invariant by the $L^2$-semigroup. Hence the claim follows from~\cite[Proposition II.1.7]{engel99}.

(iii) For $f \in \dom(\mL_{C^0})$ one has $f \circ \pi \in \dom(\Lhat)$ and 
$$
\Lhat(f\circ\pi)(x, v) = v1_{\{0 < x < L\}} f'(x).
$$
Hence, for $f, g \in \dom(\mL_{C^0})$, integrating first with respect to $v$ and then with respect to $x$ yields
\begin{align*}
    \big< \Lhat (f \circ \pi), g \circ \pi\big>_{L^2(\hat\mu)}
    &= \frac{\omega}{2 + \omega L} \int_0^L g(x) \left( \frac14 (2f'(x)) + \frac12 (0f'(x)) + \frac14 (-2f'(x)) \right) \diff x = 0.
\end{align*}
Again integrating first with respect to $v$ and then $x$ yields
\begin{align*}
    \frac12\big< \Lhat(f \circ \pi), \Lhat(g\circ \pi)\big>_{L^2(\hat\mu)}
    &= \frac\omega{2+\omega L} \int_0^L f'(x) g'(x) \diff x \\
    &= \frac\omega{2+\omega L} \left( f'(L)g(L) - f'(0)g(0) - \int_0^L f''g \diff x\right) \\
    &= \frac\omega{2+\omega L} \left( -\frac1\omega f''(L)g(L) - \frac1\omega f''(0)g(0) - \int_0^L f'' g \diff x \right) \\
    &= \int_{[0,L]} (-\mL f) g \diff\mu.\qedhere
\end{align*}
\end{proof}

\section{Convergence to equilibrium}\label{sec:main_results}

This section presents the variational approach to hypocoercivity using a flow Poincaré inequality and second-order lifts, extending the results in~\cite{albritton2024variational,cao23,lu22,eberle24} and unifying a broad class of processes, including several examples from the sampling and statistical mechanics literature.
The key idea is to characterise the collapsed process solely through its Dirichlet form.

\subsection{Convergence via flow Poincaré inequality}

For the transition semigroup $(P_t)_{t\geq 0}$ of a potentially non-reversible Markov process with invariant probability measure $\mu$ and generator $(\mL,\dom(\mL))$ in $L^2(\mu)$, as explained in Section~\ref{sec:lift_collapse}, the spectral gap or Poincar\'e inequality does not suffice in general to characterise the rate of convergence to equilibrium. In particular, the process may not satisfy a Poincar\'e inequality at all yet still converge exponentially fast. However, the following flow Poincar\'e inequality, amounting to a Poincar\'e inequality time-integrated over a fixed time length $T\in(0,\infty)$ along trajectories of the semigroup, is equivalent to exponential decay of the  $L^2(\mu)$-norm averaged over time $T$ along trajectories of the semigroup.

\begin{Def}[Flow Poincar\'e inequality]\label{def:FPI}
    The semigroup $(P_t)_{t\geq 0}$ satisfies a \emph{flow Poincar\'e inequality} with constant $\nu^{-1}\in(0,\infty)$ and duration $T\in(0,\infty)$ if 
    \begin{equation}
        \int_0^T\lVert P_t f\rVert_{L^2(\mu)}^2\diff t\ \leq\ \frac{1}{\nu}\int_0^T\langle P_tf,-\mL P_t f\rangle_{L^2(\mu)}\diff t
    \end{equation}
    for all $f\in L_0^2(\mu)\cap\dom(\mL)$.
\end{Def}

In fact, defining the inner product  $\langle f,g\rangle_T = \int_0^T\langle P_t f, P_tg\rangle_{L^2(\mu)}\diff t$
on $L^2(\mu)$, the flow Poincar\'e inequality is equivalent to
\begin{equation*}
    \langle f,-\mL f\rangle_T\geq \nu \langle f,f\rangle_T\qquad\text{for all }f\in L_0^2(\mu)\cap\dom(\mL).
\end{equation*}
Similarly to the reversible case, in which the Poincar\'e inequality is equivalent to the decay of the $L^2$-norm, the flow Poincaré inequality is equivalent to decay of time-averaged $L^2$-norms.

\begin{Thm}[Exponential decay]\label{thm:decay}
    The following assertions are equivalent:
    \begin{itemize}
        \item[(i)]  The semigroup $(P_t)_{t\geq 0}$ satisfies a flow Poincaré inequality with constant $\nu^{-1}$ and duration $T$.
        \item[(ii)] For all $f\in L_0^2(\mu)$ and $t\geq 0$ one has
    \begin{equation}\label{eq:decay}
        \int_t^{t+T}\lVert P_s f\rVert_{L^2(\mu)}^2\diff s\  \leq \ e^{-2\nu t}\int_0^T\lVert P_s f\rVert_{L^2(\mu)}^2\diff s.
    \end{equation}
    \end{itemize}
    Furthermore, both (i) and (ii) imply
    $$
        \| P_tf\|_{L^2(\mu)}\ \le\ Ce^{-\nu t} \| f\|_{L^2(\mu)}\qquad\text{with }C=e^{\nu T}
    $$
    for all $f\in L_0^2(\mu)$ and $t\geq 0$.
\end{Thm}

The proof of the decay estimate in Theorem~\ref{thm:decay} using variants of a Poincaré-type inequality in time and space is well-established, see e.g.\ \cite{albritton2024variational,cao23,eberle24},
whereas the converse relies on our formulation of the flow Poincar\'e inequality.

\begin{proof}[Proof of Theorem~\ref{thm:decay}]
    Assume that a flow Poincar\'e inequality with constant $\nu^{-1}$ and duration $T$ holds and let $f \in L^2_0(\mu)\cap \dom(\mL)$. Then
    \begin{align*}
        \MoveEqLeft\frac{\diff}{\diff t} \int_t^{t+T} \Vert  P_s f \Vert^2_{L^2(\mu)} \diff s = \Vert  P_{t+T} f \Vert_{L^2(\mu)}^2 - \Vert  P_t f \Vert_{L^2(\mu)}^2 \\
        &= \int_t^{t + T} \frac{\diff}{\diff s} \Vert  P_s f \Vert^2_{L^2(\mu)} \diff s = 2 \int_t^{t+T} \big\langle  P_s f, \mL  P_s f\big\rangle_{L^2(\mu)} \diff s.
    \end{align*}
    By the semigroup property, the flow Poincaré inequality implies
    \begin{equation*}
        \int_0^T\langle P_{t+s}f,\mL P_{t+s}f\rangle_{L^2(\mu)}\diff s\leq-\nu\int_0^T\lVert P_{t+s}f\rVert_{L^2(\mu)}^2\diff s
    \end{equation*}
    so that
    $$
    \frac{\diff}{\diff t} \int_t^{t+T} \Vert  P_s f \Vert^2_{L^2(\mu)} \diff s \le -2\nu \int_t^{t+T} \Vert  P_s f \Vert^2_{L^2(\mu)} \diff s. 
    $$
    Grönwall's inequality yields
    $$
    \int_t^{t+T} \Vert  P_s f \Vert^2_{L^2(\mu)} \diff s \le e^{-2\nu t} \int_0^{T} \Vert  P_s f \Vert^2_{L^2(\mu)} \diff s.
    $$
    The claim then follows since $\dom(\mL)$ is dense in $L^2(\mu)$.

    Conversely, assume that \eqref{eq:decay} holds and let $f\in L_0^2(\mu)\cap \dom(\mL)$. Then
    \begin{align*}
        \MoveEqLeft\int_0^T\langle P_tf,-\mL P_t f\rangle_{L^2(\mu)}\diff t = \lim_{h\to 0}\frac{1}{h}\int_0^T\langle P_t f, P_t f -  P_{t+h}f\rangle_{L^2(\mu)}\diff t\\
        &\geq\lim_{h\to 0}\frac{1}{h}(1-e^{-\nu h})\int_0^T\lVert  P_t f\rVert_{L^2(\mu)}^2\diff t = \nu\int_0^T\lVert  P_t f\rVert_{L^2(\mu)}^2\diff t,
    \end{align*}
    which yields the claim.

    Finally, if \eqref{eq:decay} holds, then for $t\geq T$,
    \begin{equation*}
        \lVert P_t f\rVert_{L^2(\mu)}^2\leq\frac{1}{T}\int_{t-T}^t\lVert P_sf\rVert_{L^2(\mu)}^2\diff s\leq e^{-2\nu(t-T)}\frac{1}{T}\int_0^T\lVert P_s f\rVert_{L^2(\mu)}^2\diff s \leq e^{-2\nu (t-T)}\lVert f\rVert_{L^2(\mu)}^2,
    \end{equation*}
    which yields the claim. For $t<T$, the claim is trivially satisfied since $e^{-\nu (t-T)}\geq 1$.
\end{proof}

\begin{Rem}
    A Poincar\'e inequality in a suitable extended space-time sense was first introduced in \cite{albritton2024variational} to obtain exponential decay for the kinetic Fokker-Planck equation, and later applied to study the convergence rate of Langevin dynamics~\cite{cao23} and some piecewise deterministic Markov processes~\cite{lu22}. The flow Poincaré inequality in Definition~\ref{def:FPI} differs from the space-time Poincaré inequality as in \cite{albritton2024variational,cao23,lu22}, as it is only valid along trajectories $t\mapsto  P_tf$ of the semigroup. However, this is not a restriction, since it is specifically intended for application to solutions of the corresponding Kolmogorov backward equation, as in the proof of Theorem~\ref{thm:decay}. More importantly, this formulation allows for its extension to a broad class of refreshment mechanism, which we will make use of in Section~\ref{sec:rtp}, and avoids the use of the $H^{-1}$-norm in the velocity that is present in the formulation of a space-time Poincaré inequality as in \cite{albritton2024variational,cao23}.
\end{Rem}

\subsection{Flow Poincaré inequality for second-order lifts}\label{ssec:FPIlift}

We now demonstrate how to prove flow Poincar\'e inequalities for non-reversible Markov processes using the framework of second-order lifts.
As in Section~\ref{sec:lift_collapse}, we consider a time-homogenous Markov process with state space $\Shat = \S\times\V$, where $\S\subseteq\R^d$, and invariant probability measure $\hat\mu$ and a reversible diffusion on $\S$ with invariant measure $\mu = \hat\mu\circ\pi^{-1}$, where $\pi(x,v)=x$ is the projection from $\Shat$ to $\S$. The associated transition semigroups acting on $L^2(\hat\mu)$ and $L^2(\mu)$ and their generators are denoted by $(\hat P_t)$ and $(P_t)$, as well as $(\Lhat,\dom(\Lhat))$ and $(\mL,\dom(\mL))$, respectively. 
Disintegrating $\hat\mu(\diff x\diff v) = \mu(\diff x)\kappa_x(\diff v)$ yields the conditional expectation
\begin{equation*}
    \Pi_v f(x, v) = \int_\V f(x, w) \,\kappa_x(\diff w).
\end{equation*}
We split the generator of $(\hat P_t)$ into 
\begin{equation}\label{eq:Lhatsplitting}
    \Lhat f \ =\ \Ld f + \gamma\Lv f\qquad\text{for all }f\in\dom(\Lhat),
\end{equation}
where $\gamma>0$, and  $(\Ld,\dom(\Ld))$ and $(\Lv,\dom(\Lv))$ are two linear operators on $L^2(\hat\mu)$ such that $\dom(\Lhat)\subseteq\dom(\Ld)\cap\dom(\Lv)$.
The operators $\Ld$ and $\Lv$ are usually interpreted as transport and refreshment terms, respectively. We denote by 
$$\E_v(f,g) = \langle f,-\Lv g\rangle_{L^2(\hat\mu)}$$
the velocity Dirichlet form and set $\E_v(f) = \E_v(f,f)$.

\begin{Asm}\label{asm:A}
    The operator $(\Ld,\dom(\Ld))$ is a lift of $(\mL,\dom(\mL))$ that is weakly antisymmetric, i.e.\ there exists a core $C$ of $\mL$ such that 
    \begin{align*}
        f\circ\pi\in\dom(\Ld^*)\qquad\text{and}\qquad\Ld^*(f \circ \pi) = -\Ld (f \circ \pi)
    \end{align*}
    for all $f\in C$.
\end{Asm}

\begin{Asm}\label{asm:B}
     The operator $(\mL,\dom(\mL))$ has purely discrete spectrum on $L^2(\mu)$ and satisfies a Poincaré inequality with constant $\frac{1}{m}$, i.e.\
\begin{align}\label{eq:xPoincare}
    \Vert f - \mu(f)\Vert_{L^2(\mu)}^2\ \le\ \frac1m \mathcal E(f)
\end{align}
for all $f \in \dom(\mathcal E)$.
\end{Asm}
\begin{Asm}\label{asm:C}
\begin{enumerate}[(i)]
    \item There exist constants $C_1,C_2 > 0$ such that
    \begin{equation}\label{eq:C1}
        \big|\big< \Ld (g  \circ \pi), \Ld (f - \Pi_v f)\big>_{L^2(\hat\mu)}\big|\ \le\ C_1\, \Vert 2 \mathcal L g\Vert_{L^2(\mu)}  (\E_v(f))^{1/2},\qquad \tag{C1}
    \end{equation}
    and 
    \begin{equation}\label{eq:C2}
            \big|\big<\Ld (g \circ \pi),\Lv f \big>_{L^2(\hat\mu)}\big| \ \le\ C_2\, \Vert 2 \mathcal L g\Vert_{L^2(\mu)}  (\E_v(f))^{1/2}, \tag{C2}
    \end{equation}
    for all functions $f\in\dom(\Lhat)$ with $\Pi_v f \in \dom(\Ld)$ and $g$ in a core of $\mL$.
    
    \item The operator $\Lv$ leaves $\kappa_x$ invariant, i.e.\
    \begin{equation*}
        \int_\V \Lv f(x, v) \diff \kappa_x(v) = 0\qquad\text{for all } x \in \mathcal S\text{ and }f\in\dom(\Lhat),
    \end{equation*}
    and there exists a constant $m_v>0$ such that 
    \begin{equation}\label{eq:mvassu}
        \int_0^T\Vert \hat P_tf - \Pi_v\hat P_tf \Vert_{L^2(\hat\mu)}^2\diff t\ \le\ \LvPoincareConstant \int_0^T\mathcal E_v(\hat P_tf)\diff t
    \end{equation}
    for all $f\in\dom(\Lhat)$.
\end{enumerate}

\end{Asm}

Let us comment on these assumptions before turning to the main result in Theorem~\ref{thm:FPIlift}.

\begin{Rem}\label{rem:assumptions}
    \begin{enumerate}[(i)]
        \item 
        A splitting as in \eqref{eq:Lhatsplitting} into two distinct dynamics, interpreted as a transport and velocity refreshment, is natural in many examples. For instance, in the case of Langevin dynamics or randomised Hamiltonian Monte Carlo, one chooses the transport operator
        \begin{equation*}
            \Ld f(x,v) = v\cdot\nabla_xf(x,v) - \nabla U(x)\cdot\nabla_vf(x,v)
        \end{equation*}
        to be the generator of the Hamiltonian flow associated to the Hamiltonian $H(x,v) = U(x)+\frac{1}{2}|v|^2$. The velocity refreshment operator $\Lv$ could be the full refreshment $\Lv = \Pi_v-I$ with respect to the conditional distribution $\kappa_x$ but much more general mechanisms are possible, as discussed in Section~\ref{sec:rtp}. Regarding the jamming RTP process, the transport and velocity refreshments terms correspond to the run-and-tumble mechanisms, respectively, with the latter differing from a complete refreshment.

        \item \label{rem:Ld_is_also_a_lift}
        In Assumption~\ref{asm:A}, if $\Lv$ only acts on $v$ in the sense that there exists a core $C$ of $\mL$ such that $\Lv(f \circ \pi) = 0$ for all $f\in C$, then $\Lhat$ is a lift of $\mL$ if and only if $\Ld$ is a lift of $\mL$.

        \item 
        The Assumption~\ref{asm:B} of a purely discrete spectrum, which was also considered in \cite{cao23,eberle24}, allows for simple calculations using a spectral decomposition that enables us to prove a divergence lemma not limited to the generator of an overdamped Langevin diffusion. This technical assumption can be relaxed, see \cite{brigati23}.
        The Poincaré inequality \eqref{eq:xPoincare} is equivalent to existence of a spectral gap $m=\Gap(\mL)$ and to exponential decay of the semigroup $(P_t)$ generated by $\mL$ with rate $m$. It thus allows us to directly compare the decay of $(\hat P_t)$ to that of $(P_t)$. In the absence of a spectral gap, Assumption~\ref{asm:A} can however also be relaxed to a weak Poincaré inequality, yielding subexponential convergence to equilibrium, see \cite{bsww24}.

        \item \label{rem:C1_with_transpose} 
        Assumption~\eqref{eq:C1} can often be verified in the following way in applications. If $\Ld (g \circ \pi) \in \dom(\Ld^*)$ for all $g$ in  a core of $\mL$, the second order lift condition becomes
        \begin{align*}
        \Pi_v \Ld^*\Ld (g \circ \pi) & = -(2 \mathcal L g) \circ \pi ,\qquad\text{
        and thus}\\
           \big|\big< \Ld (g  \circ \pi), \Ld (f - \Pi_v f)\big>_{L^2(\hat\mu)} \big|
            &\le \Vert (I - \Pi_v)\Ld^* \Ld (g  \circ \pi) \Vert_{L^2(\hat\mu)} \Vert f - \Pi_v f\Vert_{L^2(\hat\mu)}\\
            &\le \LvInvPoincareConstant^{-1/2}\Vert \Ld^* \Ld (g  \circ \pi) + (2 \mathcal L g) \circ \pi \Vert_{L^2(\hat\mu)} (\E_v(f))^{1/2}
        \end{align*}
        Hence \eqref{eq:C1} is satisfied if
        \begin{align*}
            \Vert \Ld^* \Ld (g  \circ \pi) + (2 \mathcal L g) \circ \pi \Vert_{L^2(\hat\mu)}\le C_1m_v^{1/2} \Vert 2 \mathcal L g\Vert_{L^2(\mu)},
        \end{align*}
        i.e.\ if the $L^2(\mu)$-norm of the standard deviation of $\Ld^*\Ld (g\circ\pi)$ with respect to $\kappa_x$ is bounded by the $L^2(\mu)$-norm of its expectation $-2\mL g$, since then .
        \item \label{rem:assumption_C_full_refresh}
        Assumption \eqref{eq:C2} can be verified immediately if $\Lv$ is a bounded operator, since then there exists a constant $c>0$ such that $\lVert \Lv f\rVert_{L^2(\hat\mu)}\leq c(\E_v(f))^{1/2}$.       
        By the Cauchy-Schwarz inequality and the second-order lift condition \eqref{eq:deflift2},
        \begin{equation*}
            \big|\big\langle\Ld(g\circ\pi),\Lv f\big\rangle_{L^2(\hat\mu)}\big| \leq c (2\E(g))^{1/2}(\E_v(f))^{1/2} \leq \frac{c}{\sqrt{2m}}\lVert 2\mL g\rVert_{L^2(\mu)}(\E_v(f))^{1/2},
        \end{equation*}
        so that \eqref{eq:C2} holds with $C_2 = \frac{c}{\sqrt{2m}}$. In particular, if $\Lv = \Pi_v - I$ is the full refreshment with respect to the conditional distribution $\kappa_x(\diff v)$, Assumptions \eqref{eq:C2} and \eqref{eq:mvassu} hold with $C_2  = \frac{1}{\sqrt{2m}}$ and $m_v=1$.

        \item 
        In Assumption~\ref{asm:C}, \eqref{eq:mvassu} resembles a flow Poincar\'e inequality where the left-hand side is only centred in the velocity variable, and the right-hand side uses the velocity Dirichlet form.
        The inequality in particular holds if the refreshment term satisfies a Poincaré inequality in $v$, i.e.\ if
        \begin{equation*}
            \lVert f-\Pi_vf\rVert_{L^2(\hat\mu)}^2 \leq \frac{1}{m_v}\E_v(f)
        \end{equation*}
        for all $f\in\dom(\Lhat)$. This is usually the case, as, in line with the general principles of hypocoercivity, the dynamics in $v$ has strong relaxation properties that are then propagated to the variable $x$ by the rest of the dynamics. Nonetheless, the weaker time-integrated version \eqref{eq:mvassu} may be of use in applications where the refreshment term is degenerate.

        \item Assumptions~\ref{asm:A}, \ref{asm:B} and \ref{asm:C}(ii) are closely related to Assumptions (H1), (H2), and (H3) in the seminal work \cite{DMS15}, while Assumption~\ref{asm:C}(i) ensuring good interaction between $\Ld$ and $\Lv$ is reminiscent of Assumption (H4). The connection between second-order lifts and the approach to hypocoercivity in \cite{DMS15} is detailed in \cite{BLW24}.
    \end{enumerate}
\end{Rem}

We can now establish the following flow Poincar\'e inequality.

\begin{Thm}[Flow Poincaré inequality for lifts]\label{thm:FPIlift}
    Let Assumptions~\ref{asm:A}, \ref{asm:B} and \ref{asm:C} hold. Then there exists a universal constant $C > 0$ such that for any $T\in (0,\infty)$ and $f\in \dom(\Lhat)\cap L_0^2(\hat\mu)$,
    \begin{align*}
        \int_0^T\lVert\hat P_t f\rVert_{L^2(\hat\mu)}^2\diff t\ \le\ C\Big(\gamma^2C_2^2+C_1^2 + \LvPoincareConstant \Big(1 + \frac1{mT^2}\Big)\Big) \int_0^T\E_v(\hat P_t f)\diff t.
    \end{align*}
    In particular, $(\hat P_t)_{t\geq 0}$ satisfies a flow Poincar\'e inequality with duration $T$ and constant
    \begin{equation}\label{eq:nu}
        \frac{1}{\nu} = C\gamma^{-1}\Big(\gamma^2C_2^2+C_1^2 + \LvPoincareConstant \Big(1 + \frac1{mT^2}\Big)\Big).
    \end{equation}
\end{Thm}

We defer the proof to Section~\ref{ssec:FPIproof}.
The constant $C$ in Theorem~\ref{thm:FPIlift} does not depend on any parameters and can easily be made explicit.
By Theorem~\ref{thm:decay}, Theorem~\ref{thm:FPIlift} immediately yields the following corresponding decay estimate. 

\begin{corollary}\label{coro:decaylift}
    Let Assumptions~\ref{asm:A}, \ref{asm:B} and \ref{asm:C} hold. Then the semigroup $(\hat P_t)$ generated by $(\Lhat,\dom(\Lhat))$ satisfies
    \begin{equation*}
        \lVert\hat P_tf\lVert_{L^2(\hat\mu)}\ \leq\ Ce^{-\nu t}\lVert f\rVert_{L^2(\hat\mu)}\qquad\text{for all }t\geq 0
    \end{equation*}
    and $f\in L_0^2(\hat\mu)$, where the rate $\nu$ is given by \eqref{eq:nu} and $C=e^{\nu T}$.
\end{corollary}

One of the main tools in the proof of a flow Poincaré inequality using the framework of second-order lifts is a so-called divergence lemma, see Section~\ref{ssec:divergencelemma}. There are now several different works \cite{albritton2024variational,cao23,brigati23,eberle24b} which aim at providing such a divergence lemma under general hypotheses. The Dirichlet form approach taken here enables general assumptions on the collapsed process, imposing only the Assumption~\ref{asm:B} of a spectral gap and discrete spectrum.

\begin{Rem}\label{rem:decay_rate_in_simple_setting}
    To obtain an explicit rate in Corollary~\ref{coro:decaylift} from the expression \eqref{eq:nu}, one can choose $T=m^{-1/2}$. 
    If $m_v^{-1} = O(1)$ and $C_2 = O(m^{-1/2})$, as is the case if $\Lv = \Pi_v-I$ is the full refreshment with respect to $\kappa_x$ by Remark~\ref{rem:assumptions}\eqref{rem:assumption_C_full_refresh}, this yields 
    $$
        \nu = \Omega\bigg(\frac{\gamma m}{\gamma^2 + \left(1 + C_1^2\right)m}\bigg).
    $$
    Optimising by choosing $\gamma = (1+C_1)\sqrt{m}$ yields $\nu = \Omega\big({\sqrt m}/ (1+C_1)\big)$. 
    For many examples, including Langevin dynamics, randomised HMC, and the RTP process considered here, one can choose $C_1$ of order $1$, see Section~\ref{sec:rtp}. The expression then further simplifies to
    \begin{align*}
        \nu = \Omega\bigg(\frac{\gamma m}{\gamma^2 + m}\bigg),
    \end{align*}
    and if $\gamma = \Theta (\sqrt{m})$ then $\nu = \Omega(\sqrt{m})$, i.e.\ we obtain a square-root speed-up with the optimal choice of $\gamma$.
\end{Rem}

The usual example of Langevin dynamics (e.g.\ \cite{cao23,eberle24}) serves as a useful base case in showcasing the arguments involved in applying Theorem~\ref{thm:FPIlift}.

\begin{Exa}[Langevin dynamics]\label{ex:langevin}
    Consider a probability measure $\mu$ on $\R^d$ with density proportional to $\exp(-U)$, where $U\colon\R^d\to\R$ is such that $\int_{\R^d}e^{-U(x)}\diff x<\infty$ and $\lim_{|x| \rightarrow +\infty} U(x) = +\infty$. Let us assume that $\mu$ satisfies a Poincar\'e inequality with constant $\frac{1}{m}$,
    and that the generator $(\mL,\dom(\mL))$ of the associated overdamped Langevin diffusion, see Example~\ref{ex:langevinlift}, has discrete spectrum. Furthermore, let the potential $U$ satisfy the lower curvature bound
    \begin{equation}\label{eq:curvaturebound}
        \nabla^2U(x)\geq -K m\cdot I_d\qquad\text{for all }x\in\R^d
    \end{equation}
    for some $K\geq 0$.
    Langevin dynamics is the $\R^d\times\R^d$-valued diffusion process $(X_t,V_t)$ solving the stochastic differential equation
    \begin{align*}
        \diff X_t&=V_t\diff t\,,\\
        \diff V_t&=-\nabla U(X_t)\diff t-\gamma V_t\diff t+\sqrt{2\gamma}\diff B_t\,,
    \end{align*}
    where $(B_t)_{t\geq0}$ is a $d$-dimensional Brownian motion and $\gamma\geq 0$ is the friction parameter, see e.g.\ \cite{pavliotis2014stochastic}. Its invariant measure is $\hat\mu=\mu\otimes\kappa$ with the standard normal distribution $\kappa = \mathcal{N}(0,I_d)$, and its generator in $L^2(\hat\mu)$ is given by $\Lhat f= \Ld f + \gamma\Lv f$, where 
    \begin{eqnarray*}
        \Ld f(x,v) &=& v\cdot\nabla_xf(x,v)-\nabla U(x)\cdot\nabla_vf(x,v),\\
         \Lv f(x,v)&=&-v\cdot\nabla_vf(x,v)+\Delta_vf(x,v),
    \end{eqnarray*}
    are the generators of the Hamiltonian flow associated to the Hamiltonian $H(x,v) = U(x)+\frac{1}{2}|v|^2$, and of an Ornstein-Uhlenbeck process in the velocity, respectively.
    Let us check Assumptions \ref{asm:A}--\ref{asm:C}.
    The generator $(\Ld,\dom(\Ld))$ of the Hamiltonian flow is antisymmetric in $L^2(\hat\mu)$ and is a lift of an overdamped Langevin diffusion with generator $(\mL,\dom(\mL))$ by \cite[Example 3(i)]{eberle24}.
    This can be checked immediately since $\Ld(f\circ\pi)(x,v) = v\cdot\nabla f(x)$ for all $f\in C_\comp^\infty(\R^d)$, so that Assumption~\ref{asm:A} is satisfied.
    Assumption~\ref{asm:B} is satisfied by our assumptions on the potential.
    It remains to check Assumption~\ref{asm:C}. For \eqref{eq:C1}, since the Ornstein-Uhlenbeck process with generator $\Lv$ leaves $\kappa$ invariant and has spectral gap $\LvInvPoincareConstant=1$, see for example \cite{BGL14}, by Remark~\ref{rem:assumptions}\eqref{rem:C1_with_transpose}, it is enough to prove
    \begin{align*}
        \Vert \Ld^* \Ld (g  \circ \pi) + (2 \mathcal L g) \circ \pi \Vert_{L^2(\hat\mu)} \le C_1 \Vert 2 \mathcal L g\Vert_{L^2(\mu)}.
    \end{align*}
    We have
    \begin{align*}
        \Ld^*\Ld (g \circ \pi)(x,v) &= -v^\top\nabla^2g(x) v - \nabla U(x) \cdot \nabla g(x)
    \end{align*}
    and $2 \mathcal L g = \Delta g + \nabla U \cdot \nabla g$, so that
    \begin{align*}
        \int_{\R^d} \left(\Ld^*\Ld (g \circ \pi)(\cdot, v) + 2 \mathcal L g\right)^2 \diff\kappa(v) &= \int_{\R^d} \big(v^\top\nabla^2g v - \Delta g\big)^2 \diff\kappa(v), \\
        &= \int_{\R^d} \big(v^\top\nabla^2g v \big)^2 \diff\kappa(v) - \left(\Delta g\right)^2,
    \end{align*}
    using the fact that $\int v^\top\nabla^2g v \diff\kappa(v) = \Delta g$.
    Furthermore, since $v_i, v_j, v_k, v_l$ are independent standard Gaussians,
    \begin{align*}
        \int_{\R^d} \big(v^\top\nabla^2g v \big)^2 \diff\kappa(v) &= \sum_{i, j, k, l = 1}^d (\partial_{ij} g) (\partial_{kl} g) \int_{\R^d} v_i v_j v_k v_l \diff\kappa(v) \\
        &=(\Delta g)^2 + 2\Vert\nabla^2g\Vert_F^2,
    \end{align*}
    where $\Vert\nabla^2g\Vert_F^2=\sum_{i, j = 1}^d (\partial_{ij} g)^2$ is the squared Frobenius norm of $\nabla ^2g$.
    By Bochner's formula, using the lower curvature bound \eqref{eq:curvaturebound}, we have
    \begin{align*}
        \big\Vert\Vert\nabla^2g\Vert_F\big\Vert_{L^2(\mu)}^2 = \Vert 2 \mathcal Lg \Vert_{L^2(\mu)}^2 - \int_{\R^d} \nabla g^\top\nabla^2U\nabla g \diff\mu\leq \Vert 2 \mathcal Lg \Vert_{L^2(\mu)}^2+2Km\,\mathcal E(g),
    \end{align*}
    Since $\E(g)\leq\frac{1}{m}\lVert\mL g\rVert_{L^2(\mu)}^2$, overall we obtain
    \begin{align*}
        \Vert \Ld^* \Ld (g  \circ \pi) + (2 \mathcal L g) \circ \pi \Vert_{L^2(\hat\mu)}^2 = 2\big\Vert\Vert\nabla^2g\Vert_F\big\Vert_{L^2(\mu)}^2 \leq (2+K)\Vert 2 \mathcal Lg \Vert_{L^2(\mu)}^2
    \end{align*}
    and (\ref{eq:C1}) holds with $C_1 = \sqrt{2+K}$.
    Finally, Assumption~\eqref{eq:C2} is satisfied with $C_2=\frac{1}{\sqrt{2m}}$ since
    \begin{equation*}
        \big<\Ld (g \circ \pi),\Lv f\big>_{L^2(\hat\mu)} = \left< v \cdot \nabla_x g,-\nabla_v^* \nabla_v f\right>_{L^2(\hat\mu)} =  \left<\nabla_xg,-\nabla_vf\right>_{L^2(\hat \mu)}.
    \end{equation*}
    
    Hence, by Corollary~\ref{coro:decaylift}, when choosing $T=m^{-1/2}$, the transition semigroup $(\hat P_t)$ associated to Langevin dynamics is exponentially contractive in $T$-average with rate
    \begin{equation*}
        \nu = \Omega\left( \frac{\gamma m}{\gamma^2 + (1+K)m} \right).
    \end{equation*}
    Choosing $\gamma = \sqrt{(1+K)m}$ yields the rate $\nu = \Omega(\sqrt{m}/\sqrt{1+K})$. We thus recover the result of \cite[Theorem 1]{cao23}.
\end{Exa}

\subsection{Divergence lemma and space-time property of lifts}\label{ssec:divergencelemma}

In the classical elliptical setting, a Poincaré inequality can be established by solving the Poisson equation $-\mL g=f$ with explicit bounds on $g$. Similarly, in the hypoelliptic setting, a generalised time-dependent Poisson equation may be used to obtain a flow Poincaré inequality. The divergence Lemma~\ref{lem:divergence_lemma} provides solutions to such an equation with explicit bounds on the solution and is at the heart of the proof of the flow Poincaré inequality Theorem~\ref{thm:FPIlift}. Different versions of this divergence lemma have been studied recently by \cite{cao23,brigati23,eberle24b, Lehec25}, and a related divergence equation with different boundary conditions is of fundamental importance in fluid mechanics \cite{galdi11}. We take an approach using Dirichlet forms, adapting the proof of \cite{eberle24b} and enabling minimal structural assumptions on $\mL$. 

Let us assume that $\S\subseteq\R^d$ and that $C^\infty(\S)\cap L^2(\mu)\subseteq\dom(\mL)$. For $g\in L^2(\lambda\otimes\mu)$ such that $g(t,\cdot)\in\dom(\E)$ for $\lambda$-almost every $t\in[0,T]$, denote by
\begin{equation*}
    \E_T(g) = \int_0^T \E(g(t,\cdot))\diff t
\end{equation*}
the time-integrated Dirichlet form.
Let the Sobolev space $H^{1,2}_\D(\lambda\otimes\mu)$ with Dirichlet boundary conditions in time be the closure of 
\begin{equation*}
    \{u\in C^\infty([0,T]\times\S)\cap L^2(\lambda\otimes\mu)\colon u(0,\cdot)=u(T,\cdot)=0\}
\end{equation*}
with respect to the norm given by 
\begin{equation*}
    \Vert u\Vert_{1,2} = \myVert{u}+\E_T(u)^{1/2} + \myVert{\partial_tu},
\end{equation*}
and similarly let $H^{2,2}_\D(\lambda\otimes\mu)$ be the closure of the same set with respect to
\begin{equation*}
    \Vert u\Vert_{2,2} = \myVert{u}+\E_T(u)^{1/2}+\myVert{\partial_tu}+\myVert{\mL u}+\myVert{\partial_{tt}u}+\E_T(\partial_tu)^{1/2}.
\end{equation*}
In particular, if $u\in H^{2,2}_\D(\lambda\otimes\mu)$, then $u(t,\cdot)\in\dom(\mL)$ for $\lambda$-a.e.\ $t\in[0,T]$ since $(\mL,\dom(\mL))$ is closed. 
The assumption that $\S\subseteq\R^d$ can be relaxed to $\S$ being a Riemannian manifold with boundary, see \cite{eberle24b}.

\begin{Lem}[Divergence lemma]\label{lem:divergence_lemma}
    Suppose that Assumption~\ref{asm:B} is satisfied. 
    Then for any $T\in (0,\infty)$, there exists a constant $C\in (0,\infty)$ such that for any $f\in L_0^2(\lambda\otimes\mu)$ there are functions
    \begin{equation*}
        h\in H_\D^{1,2}(\lambda\otimes\mu)\qquad\text{and}\qquad g\in H_\D^{2,2}(\lambda\otimes\mu)
    \end{equation*}
    such that 
    \begin{equation*}
        \partial_th-2\mL g \ = \ f,
    \end{equation*}
    and the functions $h$ and $g$ satisfy
    \begin{align*}
    \Vert 2 \mathcal L g \Vert_{L^2(\lambda \otimes \mu)} &\ \le\ C \Vert f \Vert_{L^2(\lambda \otimes \mu)},\\
    \sqrt{2 \mathcal E_T(h)} &\ \le\ C \Vert f \Vert_{L^2(\lambda \otimes \mu)}, \\
    \sqrt{2 \mathcal E_T(\partial_t g)} &\ \le\ C(1 + T^{-1}m^{-1/2}) \Vert f \Vert_{L^2(\lambda \otimes \mu)}.
    \end{align*}
\end{Lem}

The proof of Lemma~\ref{lem:divergence_lemma} is delayed until Section~\ref{ssec:DivLemmaProof}. 
The final ingredient for the proof of the flow Poincaré inequality is the following space-time property of lifts. We consider the operator $(A,\dom(A))$ on $L^2([0,T]\times\Shat,\lambda\otimes\hat\mu)$ defined by
\begin{equation}\label{eq:defA}
    Af=-\partial_tf+\Ld f
\end{equation}
with domain consisting of all functions  $ f\in L^2(\lambda\otimes\hat\mu)$ such that $f(\cdot,x,v)$ is absolutely continuous for $\hat\mu$-a.e.\ $(x,v)\in\Shat$ with $\partial_t f\in L^2(\lambda\otimes\hat\mu)$, and $f(t,\cdot)\in\dom(\Ld)$ for $\lambda\textup{-a.e.\ }t\in[0,T]$ with $\Ld f\in L^2(\lambda\otimes\hat\mu)$.

\begin{Lem}[Space-time property of lifts]\label{lem:space_time_property_of_lifts}
Let $(\Ld,\dom(\Ld))$ be a lift of $(\mL,\dom(\mL))$. Then, for any $f,g,h\in L^2(\lambda\otimes\mu)$ such that $f\circ\pi\in\dom(A)$ and $g(t,\cdot)\in\dom(\mL)$ for a.e.\ $t\in[0,T]$ with $\mL g\in L^2(\lambda\otimes\mu)$, we have
\begin{align*}
    \big<A (f \circ \pi), h \circ \pi + \Ld (g \circ \pi)\big>_{L^2(\lambda\otimes\hat\mu)} = -\myAngle{h, \partial_t f} + 2 \mathcal E_T (f, g).
\end{align*}
\end{Lem}
\begin{proof}
    This follows directly from the defining properties \eqref{eq:deflift1} and \eqref{eq:deflift2} of second-order lifts, see \cite[Lemma 18]{eberle24}. 
\end{proof}

\subsection{Proof of the flow Poincaré inequality for second-order lifts}\label{ssec:FPIproof}

We now prove Theorem~\ref{thm:FPIlift}.
The proof is similar to that of \cite[Theorem 23]{eberle24}. However, bounding $\int_0^T\Vert \hat P_tf\Vert_{L^2(\hat\mu)}^2\diff t$ directly by the time-integrated velocity Dirichlet form $\int_0^T\E_{v}(\hat P_tf)\diff t$ in this general setting requires more refined estimates.

\begin{proof}[Proof of Theorem~\ref{thm:FPIlift}] Let $f_0\in L^2_0(\hat\mu)$ and let $f\in L^2(\lambda\otimes\hat\mu)$ be defined by
\begin{equation*}
    f(t,x,v) = (\hat P_t f_0)(x,v).
\end{equation*}
Further let $\tilde f \in L^2(\lambda \otimes \mu)$ be such that $\Pi_v f = \tilde f \circ \pi$. Let $\tilde f = \partial_th-2\mL g$ with 
$h$ and $g$ as in Lemma~\ref{lem:divergence_lemma}. Because $\Pi_v$ is an orthogonal projection, we have
$$
\myVertHat{f}^2 = \Vert \tilde f \Vert_{L^2(\lambda \otimes \mu)}^2 + \myVertHat{f - \Pi_v f}^2.
$$
Furthermore, by an integration by parts and the space-time property of lifts in Lemma~\ref{lem:space_time_property_of_lifts},
\begin{align}\label{eq:FlowPoincareProof}
    \Vert \tilde f \Vert_{L^2(\lambda \otimes \mu)}^2 &=\big<\partial_t h - 2 \mathcal L g, \tilde f\big>_{L^2(\lambda\otimes\mu)} =-\big<h, \partial_t \tilde f\big>_{L^2(\lambda\otimes\mu)} + 2\mathcal E_T(\tilde f, g) \nonumber\\
    &= \big< h \circ \pi + \Ld (g \circ \pi), A(\tilde f \circ \pi)\big>_{L^2(\lambda \otimes \hat \mu)}\\
    &= \big< h \circ \pi + \Ld (g \circ \pi), A f \big>_{L^2(\lambda \otimes \hat \mu)} + \big< h \circ \pi + \Ld (g \circ \pi), A(\Pi_v f - f)\big>_{L^2(\lambda \otimes \hat \mu)}\nonumber
\end{align}
The two summands on the right-hand side of \eqref{eq:FlowPoincareProof} can be bounded separately. 

Firstly, we have $Af = -\partial_t f + \Ld f = -\gamma \Lv f$, so that $\langle h\circ\pi,Af\rangle_{L^2(\lambda\otimes\hat\mu)}=0$,
and thus
\begin{align*}
    \MoveEqLeft\big<h \circ \pi + \Ld (g \circ \pi), Af\big>_{L^2(\lambda\otimes\hat\mu)} = -\gamma \big<\Ld (g \circ \pi),\Lv f\big>_{L^2(\lambda\otimes\hat\mu)} \\
    &\le  \gamma C_2\int_0^T\lVert2\mL g\rVert_{L^2(\mu)}(\E_v(f))^{1/2}\diff t \le  \gamma C\cdot C_2  \int_0^T\Vert\tilde f\Vert_{L^2(\mu)}(\mathcal E_{v}(f))^{1/2}\diff t\\
    &\leq \gamma C\cdot C_2\lVert\tilde f\rVert_{L^2(\lambda\otimes\mu)}\Big(\int_0^T \E_v(f)\diff t\Big)^{1/2},
\end{align*}
using Assumption~(\ref{eq:C2}) for the first inequality and Lemma~\ref{lem:divergence_lemma} for the second.

The second summand in \eqref{eq:FlowPoincareProof} can be estimated using \eqref{eq:defA}.
The Dirichlet boundary conditions in time of $h$ and $g$ give
\begin{align*}
    \MoveEqLeft\big<h \circ \pi + \Ld(g \circ \pi), -\partial_t (\Pi_v f - f)\big>_{L^2(\lambda\otimes\hat\mu)} = \big<\partial_t h \circ \pi + \partial_t \Ld(g \circ \pi),\Pi_v f - f\big>_{L^2(\lambda\otimes\hat\mu)} \\
    &= \big<\partial_t \Ld(g \circ \pi),\Pi_v f - f\big>_{L^2(\lambda\otimes\hat\mu)}\le \Vert\partial_t \Ld (g \circ \pi)\Vert_{L^2(\lambda\otimes\mu)} \myVertHat{f - \Pi_v f}. 
\end{align*}
Furthermore, using the fact that $\partial_t$ commutes with $\Ld(\cdot \circ \pi)$ and the second-order lift property, we get
\[
	\Vert\partial_t \Ld (g \circ \pi)\Vert_{L^2(\lambda\otimes\hat\mu)}  = (2\E_T(\partial_tg))^{1/2}\leq C\Big(1+\frac{1}{\sqrt{m}T}\Big)\lVert\tilde f\rVert_{L^2(\lambda\otimes\mu)}.
\]
We deduce that
\begin{align*}
	\MoveEqLeft\big<h \circ \pi + \Ld(g \circ \pi), -\partial_t (\Pi_v f - f)\big>_{L^2(\lambda\otimes\hat\mu)}\\
    &\leq \frac{C}{\sqrt{m_v}}\Big(1+\frac{1}{\sqrt{m}T}\Big)\Big(\int_0^T\E_v(f)\diff t\Big)^{1/2}\lVert \tilde f\rVert_{L^2(\lambda\otimes\mu)}^2.
\end{align*}
Finally,
\begin{align*}
\MoveEqLeft\big<h \circ \pi + \Ld(g \circ \pi), \Ld (f - \Pi_v f)\big>_{L^2(\lambda\otimes\hat\mu)} \\
&= -\big<\Ld (h \circ \pi), f - \Pi_v f\big>_{L^2(\lambda\otimes\hat\mu)} + \big<\Ld(g \circ \pi), \Ld (f - \Pi_v f)\big>_{L^2(\lambda\otimes\hat\mu)}\\
&\le \Vert\Ld (h \circ \pi)\Vert_{L^2(\lambda\otimes\hat\mu)} \myVertHat{f - \Pi_v f}+ C_1\int_0^T\lVert2 \mathcal L g\rVert_{L^2(\mu)}(\E_v(f))^{1/2}\diff t\\
&\leq \Big(\frac{1}{\sqrt{m_v}}\big(2\E_T(h)\big)^{1/2} + C_1\lVert2\mL g\rVert_{L^2(\lambda\otimes\mu)}\Big)\Big(\int_0^T\E_v(f)\diff t\Big)^{1/2}\\
&\leq C\Big(\frac{1}{\sqrt{m_v}}+C_1\Big)\lVert\tilde f\rVert_{L^2(\lambda\otimes\mu)}\Big(\int_0^T\E_v(f)\diff t\Big)^{1/2}
\end{align*}
using the weak antisymmetry of $\Ld$, Assumption \eqref{eq:C1} and Lemma~\ref{lem:divergence_lemma}.

Combining everything finally leads to
\begin{align*}
    \myVertHat{f}^2&\le \myVertHat{\Pi_v f - f}^2 \\
    &\qquad+ C^2\Big( \gamma C_2 + \frac{1}{\sqrt{m_v}}\Big(1+\frac{1}{\sqrt{m}T}\Big) + \Big(\frac{1}{\sqrt{m_v}}+C_1\Big) \Big)^2 \int_0^T\E_v(f)\diff t \\
    &\le \tilde C \Big(\gamma^2C_2^2+C_1^2 + \frac{1}{m_v}\Big(1+\frac{1}{mT^2}\Big)\Big)\int_0^T\E_v(f)\diff t,
\end{align*}
where $\tilde C$ denotes an absolute constant, as claimed.

\end{proof}

\subsection{Proof of the divergence lemma}\label{ssec:DivLemmaProof}

We now prove Lemma~\ref{lem:divergence_lemma}. While some differences arise due to the fact that we estimate slightly different quantities, the proof closely follows that of \cite[Theorem 5]{eberle24b}, which is a careful refinement of \cite[Lemma 2.6]{cao23}. For the convenience of the reader, we provide an overview of the proof, omitting the details where they are identical to \cite{eberle24b}.

\begin{proof}[Proof of Lemma~\ref{lem:divergence_lemma}]
    Consider an orthonormal basis \( \{e_k : k \ge 0\} \) of eigenfunctions of \( \mathcal L \) with eigenvalues \( -\alpha_k^2 \), where \( e_0 \) is constant and \( \alpha_0 = 0 \). In particular, \( \inf \{ \alpha_k^2\colon k > 0 \} = m \) is the spectral gap of $\mathcal L$. Then \( \{H_k^a : k \ge 0\} \cup \{H_k^s : k > 0\} \) given by
    \begin{align*}
    	H_0^a(t, x) &= (t - (T - t)) e_0(x), \\
    	H_k^a(t, x) &= \big(e^{-\alpha_k t} - e^{-\alpha_k(T - t)}\big) e_k(x) \text{ for } k > 0, \\
    	H_k^s(t, x) &= \big(e^{-\alpha_k t} + e^{-\alpha_k(T - t)}\big) e_k(x) \text{ for } k > 0,
    \end{align*}
    defines an orthogonal basis of the space
    \begin{align*}
        \mathbb{H} = \big\{u\in L_0^2(\lambda\otimes\mu)\colon &u(\cdot,x)\in H^2((0,T))\text { for } \mu\text{-a.e.\ } x,\,u(t, \cdot) \in \text{Dom}(\mathcal L) \text{ for } \lambda\text{-a.e.\ } t\text{, and}\\
        &\partial_{tt}u+2\mL u = 0\big\}
    \end{align*}
    of mean-zero harmonic functions for $\partial_{tt}+2\mL$. We consider a simultaneous decomposition of $\mathbb{H}$ into high and low modes and symmetric and antisymmetric functions, i.e.\
    \begin{equation*}
    	L^2_0(\lambda \otimes \mu) = \mathbb H \oplus \mathbb H^\perp = \mathbb H_{l, a} \oplus \mathbb H_{l, s} \oplus \mathbb H_{h, a} \oplus \mathbb H_{h, s} \oplus \mathbb H^\perp,
    \end{equation*}
    where
    \begin{align*}
    	\mathbb H_{l, a} &= \spn \Big\{ H_k^a : \alpha_k \le \frac{2}{T} \Big\},&
    	\mathbb H_{l, s} &= \spn \Big\{ H_k^s : \alpha_k \le \frac{2}{T} \Big\}, \\
    	\mathbb H_{h, a} &= \overline{\spn}\,\Big\{ H_k^a : \alpha_k > \frac{2}{T} \Big\},&
    	\mathbb H_{h, s} &= \overline{\spn}\, \Big\{ H_k^s : \alpha_k > \frac{2}{T} \Big\}.
    \end{align*}
    Due to linearity, it suffices to prove the statement separately for right-hand sides $f$ in these subspaces and in $\mathbb{H}^\perp$.

    \textit{Case 1: } Let $f\in\mathbb{H}^\perp$. Consider the operator $(\bar\mL,\dom(\bar\mL))$ on $L^2(\lambda\otimes\mu)$ given by $\bar\mL f = \partial_{tt}f+2\mL f$ with Neumann boundary conditions in time, i.e.\
    \begin{align*}
        \dom(\bar\mL) &= \big\{f\in L^2(\lambda\otimes\mu)\colon f(t,\cdot)\in\dom(\mL)\text{ for }\lambda\text{-a.e.\ }t\in(0,T),\\
        &\qquad\qquad f(\cdot,x)\in H^{2,2}((0,T))\text{ for }\mu\text{-a.e.\ }x\in\S\text{ and }\partial_tf = 0\text{ on }\{0,T\}\times\S,\\
        &\qquad\qquad\partial_{tt}f+\mL f\in L^2(\lambda\otimes\mu)\big\}.
    \end{align*}
    By tensorisation, $\bar\mL$ satisfies a Poincaré inequality with constant $\min(2m,\frac{\pi^2}{T})$. Hence $\bar\mL$ has a spectral gap, so that
    \begin{equation*}
        -\bar\mL|_{L_0^2(\lambda\otimes\mu)}\colon\dom(\bar\mL)\cap L_0^2(\lambda\otimes\mu)\to L_0^2(\lambda\otimes\mu)
    \end{equation*}
    admits a bounded linear inverse
    \begin{equation*}
        \bar{G}\colon L_0^2(\lambda\otimes\mu)\to L_0^2(\lambda\otimes\mu)\cap\dom(\bar\mL)
    \end{equation*}
    whose operator norm is bounded above by $\max(\frac{1}{2m},\frac{T^2}{\pi^2})$. We hence set $g = \bar Gf$ and $h=-\partial_tg$. The Dirichlet boundary conditions in time for $h$ are satisfied by construction, and the Dirichlet boundary conditions for $g$ follow from an integration by parts as in \cite{eberle24b}. On $L_0^2(\lambda\otimes\mu)\cap\dom(\bar\mL)$, the operators $-\partial_{tt}$ and $-2\mL$ commute and have discrete spectrum, so that the operators norms of
    \begin{equation*}
        2\mL\bar G = 2\mL(-\partial_{tt}-2\mL)^{-1}\quad\text{and}\quad\partial_{tt}\bar G = \partial_{tt}(-\partial_{tt}-2\mL)^{-1}
    \end{equation*}
    are bounded by one.
    Since
    \begin{equation*}
        2\E_T(h) = 2\E_T(\partial_tg) = \langle -\partial_{tt}\bar Gf,-2\mL\bar Gf\rangle_{L^2(\lambda\otimes\mu)}\leq\Vert f\Vert_{L^2(\lambda\otimes\mu)}^2,
    \end{equation*}
    for right-hand sides $f\in\mathbb{H}^\perp$, we obtain the claim with the estimates
    \begin{equation*}
        \lVert 2\mL g\rVert_{L^2(\lambda\otimes\mu)}\leq \lVert f\rVert_{L^2(\lambda\otimes\mu)}\quad\text{and}\quad \sqrt{2\E_T(h)} = \sqrt{2\E_T(\partial_tg)}\leq \lVert f\rVert_{L^2(\lambda\otimes\mu)}. 
    \end{equation*}

    \textit{Case 2:} Let $f\in\mathbb{H}_{l,a}$. Here,   as in \cite{eberle24b}, we set $h(t, x) = \int_0^t f(s, x) \diff s$ and $g(t, x) = 0$. Then the boundary conditions are satisfied due to the antisymmetry of $f$ and we obtain
    \begin{eqnarray*}
        \Vert h\Vert_{L^2(\lambda\otimes\mu)}^2 &\leq &\frac{T^2}{2}\Vert f\Vert_{L^2(\lambda\otimes\mu)}^2,\\
        \Vert\mL h\Vert_{L^2(\lambda\otimes\mu)}^2 &\leq &\int_0^T t \int_0^t \frac{16}{T^4}\myVert{f(s, \cdot)}^2 \diff s\diff t \le \frac{8}{T^2} \myVert{f}^2,
    \end{eqnarray*}
    where we used that $\myVert{\mL f(s,\cdot)}\leq\frac{4}{T^2}\myVert{f(s,\cdot)}$ since $f\in\mathbb{H}_{l,a}$. By Cauchy-Schwarz, this yields $\E_T(h)\leq 2\myVert{f}$, so that, for right-hand sides $f\in\mathbb{H}_{l,a}$, we obtain the statement with the estimates
    \begin{equation*}
        \lVert 2\mL g\rVert_{L^2(\lambda\otimes\mu)}=\sqrt{2\E_T(\partial_tg)}=0\quad\text{and}\quad \sqrt{2\E_T(h)} \leq 2\lVert f\rVert_{L^2(\lambda\otimes\mu)}. 
    \end{equation*}

    \textit{Case 3:} Let $f\in\mathbb{H}_{l,s}$. Since $f$ is symmetric, it no longer necessarily integrates to $0$ in time and we cannot argue as previously. We consider the decomposition $f=f_0+f_1$ with
    \begin{align*}
    	f_0(t, x) = f(0, x) \cos\left(\frac{2\pi t}{T}\right) \qquad\text{and}\qquad
    	f_1(t, x) = f(t, x) - f_0(t, x)
    \end{align*}
    into a part $f_0$ that integrates to $0$ in time and a part $f_1$ with Dirichlet boundary values. We then set $h(t,x) = \int_0^tf_0(s,x)\diff s$ and $g(x) = Gf_1(t,\cdot)(x)$, where $G = -(2\mL)^{-1}$ on $\spn\{e_k\colon \alpha_k\leq\frac{2}{T}\}$, i.e.\ $Ge_k = \frac{1}{2\alpha_k^2}e_k$, so that indeed $f = \partial_th-2\mL g$.
    As in \cite{eberle24b}, one shows that
    \begin{equation*}
        \myVert{f_0}^2\leq e^2\myVert{f}^2\quad\text{and}\quad \myVert{2\mL g}^2=\myVert{f_1}^2\leq (1+e)^2\myVert{f}^2,
    \end{equation*}
    and the same estimate as previously yields $2\E_T(h)\leq 2\myVert{f_0}^2\leq 4e^2\myVert{f}^2$. The term $2\E_T(\partial_tg)$ can be bounded as in \cite{eberle24b}, giving
    \begin{equation*}
        2\E_T(\partial_tg) \leq \frac{(2+\sqrt{2}\pi e)^2}{2mT^2}\myVert{\partial_tf_1}^2.
    \end{equation*}
    Hence, for right-hand sides $f\in\mathbb{H}_{l,s}$, the statement is satisfied 
    with the estimates
    \begin{align*}
        \lVert 2\mL g\rVert_{L^2(\lambda\otimes\mu)} &\leq (1+e)\lVert f\rVert _{L^2(\lambda\otimes\mu)},\quad \sqrt{2\E_T(h)}\leq 2e\lVert f\rVert _{L^2(\lambda\otimes\mu)},\\
        \sqrt{2\E_T(\partial_tg)}&\leq \frac{10}{\sqrt mT}\lVert f\rVert_{L^2(\lambda\otimes\mu)}.
    \end{align*}    

    \textit{Case 4:} Let $f\in\mathbb{H}_{h,a}$. As in \cite{eberle24b}, we use the expansion $f(t,x) = \sum_{\alpha_k>\frac{\beta}{T}}b_kH_k^a(t,x)$, and it is sufficient to derive a representation
    \begin{equation*}
        H_k^a = \partial_th_k-2\mL g_k
    \end{equation*}
    for each of the basis functions $H_k^a$ with $k\geq 1$ fixed. We write  $u_k(t) = e^{-\alpha_kt}-e^{-\alpha_k(T-t)}$, so that $H_k^a(t,x) = u_k(t)e_k(x)$ and employ the ansatz 
    \begin{equation}\label{eq:ansatzuvw_a}
        u_k = \dot v_k+w_k\quad\textup{with }v_k(0)=v_k(T)=w_k(0)=w_k(T)=0\,.
    \end{equation}
    Then setting $h_k(t,x) = v_k(t)e_k(x)$ and $g_k(t,x) = \frac{1}{2\alpha_k^2}w_k(t)e_k(x)$ yields $H_k^a = \partial_th_k-2\mL g_k$ as desired. The construction of such $v_k$ and $w_k$ and the associated bounds follow as in \cite{eberle24b}, we finally obtain the claim for right-hand sides $f\in\mathbb{H}_{h,a}$ 
    with the estimates
    \begin{align*}
        \lVert 2\mL g\rVert_{L^2(\lambda\otimes\mu)} &\leq \Big(1+\frac{1}{\sqrt{3}}\Big)\lVert f\rVert _{L^2(\lambda\otimes\mu)},\quad \sqrt{2\E_T(h)}\leq \frac{1}{\sqrt{15}}\lVert f\rVert _{L^2(\lambda\otimes\mu)},\\
        \sqrt{2\E_T(\partial_tg)}&\leq 8\lVert f\rVert_{L^2(\lambda\otimes\mu)}.
    \end{align*}   
    
    \textit{Case 5:} Let $f\in\mathbb{H}_{h,s}$. This can be treated as in the previous case, the bounds even improve slightly, see \cite{eberle24b}. We obtain the claim for $f\in\mathbb{H}_{h,s}$ with the estimates
    \begin{align*}
        \lVert 2\mL g\rVert_{L^2(\lambda\otimes\mu)} &\leq \Big(1+\frac{1}{\sqrt{3}}\Big)\lVert f\rVert _{L^2(\lambda\otimes\mu)},\quad \sqrt{2\E_T(h)}\leq \frac{1}{\sqrt{15}}\lVert f\rVert _{L^2(\lambda\otimes\mu)},\\
        \sqrt{2\E_T(\partial_tg)}&\leq (5+\sqrt{2})\lVert f\rVert_{L^2(\lambda\otimes\mu)}.
    \end{align*}    

    Finally, the statement follows with constants given by $5$ times the maximum of the constants on the five subspaces.

\end{proof}

\section{Convergence of the jamming RTP process}\label{sec:rtp}

The challenges posed by the RTP process are twofold, namely the non-trivial boundary behaviour with positive time spent on the boundary, and the intricate refreshment mechanism described by the velocity jumps.
Addressing the first difficulty, we have already seen in Section~\ref{sec:rtp_introduction} that the RTP process on $[0,L]$ with parameter $\omega$ is a lift of a time-changed sticky Brownian motion on $[0,L]$ with parameter $\omega$. 
We now begin by proving the main result of convergence to equilibrium in Section~\ref{ssec:RTPconvergence} and continue in Section~\ref{ssec:RTPdomainproof} by completing the proof of Theorem~\ref{thm:invariant_measure_generator_RTP} which is necessary to characterise the $L^2$-generator of the RTP process.

\subsection{Convergence to equilibrium}\label{ssec:RTPconvergence}
We have to verify Assumptions~\ref{asm:A}--\ref{asm:C} for the RTP process.
To do so, the first step is splitting the generator
\begin{equation*}
    \Lhat f(x, v) = v 1_{\{ 0 < x < L \}} \partial_x f(x, v) + \sum_{\tilde v \in \V} \lambda_{v\tilde v} f(x, \tilde v),
\end{equation*}
of the RTP process, see \eqref{eq:RTPgenerator}, into $\Lhat = \Ld + \gamma\Lv$. By analogy with Langevin dynamics, it would be tempting to choose
$$
\Ld f(x,v) = v1_{\{0<x<L\}}\partial_xf(x,v) \qquad \text{ and } \qquad \gamma \Lv f(x,v) = \sum_{\tilde v \in \V} \lambda_{v\tilde v} f(x, \tilde v)
$$
to be the deterministic part of the dynamics and the random jumps, respectively.  Unfortunately, that splitting does not take into account the interplay with the boundaries properly and $\hat\mu$ is invariant for neither $\Ld$ nor $\Lv$. Indeed, $\gamma \Lv$ should leave
$$
    \kappa_x =
    \left\{
    \begin{array}{cl}
        \frac12 \delta_0 + \frac12 \delta_{-2} & \text{ if } x = 0, \\
        \frac14 \delta_2 + \frac12 \delta_0 + \frac14 \delta_{-2} & \text{ if } x \in (0, L), \\
        \frac12 \delta_{2} + \frac12 \delta_0 & \text{ if } x = L,
    \end{array}
    \right.
$$
invariant. Therefore, \( \gamma \Lv \) is taken to be jumps in \( v \) with the rates of Figure~\ref{fig:relative_velocity_rates} when \( x \in (0, L) \), the rates of Figure~\ref{fig:rtp_Lv_x_equal_0} when $x = 0$ and the rates of Figure~\ref{fig:rtp_Lv_x_equal_L} when $x=L$.
\begin{figure}[H]
     \centering
     \begin{subfigure}[b]{0.4\textwidth}
         \centering
         \includegraphics[width=.5\textwidth]{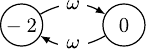}
         \caption{$x = 0$}
         \label{fig:rtp_Lv_x_equal_0}
     \end{subfigure}
     \begin{subfigure}[b]{0.4\textwidth}
         \centering
         \includegraphics[width=.5\textwidth]{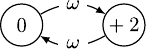}
         \caption{$x = L$}
         \label{fig:rtp_Lv_x_equal_L}
     \end{subfigure}
     \hfill
        \caption{Transition rates of $v$ under $\gamma \Lv$.}
\end{figure}
Choosing $\gamma = \omega$ leads to
\begin{align*}
	\Lv f(x, +2) &= 1_{\{x < L \}} 2(f(x, 0) - f(x, +2)) + 1_{\{x = L\}} (f(x, 0) - f(x, +2)), \\
	\Lv f(x, 0)  &= 1_{\{x > 0\}} (f(x, +2) - f(x, 0)) + 1_{\{ x < L\}} (f(x, -2) - f(x, 0)), \\
	\Lv f(x, -2) &= 1_{\{x > 0\}} 2 (f(x, 0) - f(x, -2)) + 1_{\{x = 0\}} (f(x, 0) - f(x, -2)),
\end{align*}
and 
\begin{align*}
	\Ld f(x, +2) &= 1_{\{x < L \}} (+2)\partial_x f(x, +2) + 1_{\{x = L\}}\omega (f(x, 0) - f(x, +2)), \\
	\Ld f(x, 0)  &= 1_{\{x = 0\}} \omega(f(x, +2) - f(x, 0)) + 1_{\{x = L\}}\omega (f(x, -2) - f(x, 0)), \\
	\Ld f(x, -2) &= 1_{\{x > 0\}} (-2) \partial_x f(x, -2) + 1_{\{x = 0\}}\omega (f(x, 0) - f(x, -2)).
\end{align*}
The transport term $\Ld$ hence describes motion with constant velocity $v$ in the interior $(0,L)$ combined with velocity jumps with rate $\omega$ on the boundary.

Since the RTP process on $[0,L]$ with parameter $\omega$ is a lift of a sticky Brownian motion with parameter $\omega$ by Theorem~\ref{thm:RTP_collapse}, we begin by showing that the generator of sticky Brownian motion satisfies a Poincar\'e inequality and has discrete spectrum.

\begin{Lem}[Poincaré inequality for sticky Brownian motion]\label{lem:stickypoincare} There exists $C > 0$ such that
$$
    \Vert f - \mu(f) \Vert_{L^2(\mu)}^2 \le C \frac{\omega L + (\omega L)^2}{\omega^2} \left< f, -\mL f \right>_{L^2(\mu)}
$$
for all $f \in \dom(\mL)$.
\end{Lem}

\begin{proof}

Note that
$$
    \left< f, (-\mL f) \right>_{L^2(\mu)} = \frac\omega{2+\omega L} \int_0^L (f')^2 \diff x
$$
hence the goal is to establish an inequality of the form
$$
    \Var_\mu(f) \le \frac{1}{m} \frac{\omega}{2 + \omega L} \int_0^L (f')^2 \diff x.
$$
Recall that if $\mu_p=\int \Pi_\theta \diff p(\theta)$ with  probability measures $\Pi_\theta$ and $p$ then
$$\mbox{Var}_{\mu_p}(f) = \int \mbox{Var}_{\Pi_\theta}(f) p(\diff\theta) +\mbox{Var}_p\left(\int f\diff\Pi_\theta\right).$$
Recall also that
$$\mbox{Var}_{\frac12\delta_0+\frac12\delta_L}(g)=\frac14(g(0)-g(L))^2=\frac14\left(\int_0^L g'(x)\diff x\right)^2.$$
Rather than applying this to the triple mixture we have, we apply it twice on a mixture of two components by rewriting
$$\mu=\frac{\omega L}{2+\omega L} \mathcal U_{[0,L]}+\frac{2}{2+\omega L}\left(\frac 12\delta_0+\frac12\delta_L\right),$$
where $ \mathcal U_{[0,L]}$ is the uniform distribution on $[0,L]$. Then
\begin{align*}
\mbox{Var}_\mu(f)&= \frac{\omega L}{2+\omega L}\mbox{Var}_{\mathcal U_{[0,L]}}(f)+\frac{2}{2+\omega L}\mbox{Var}_{\frac12\delta_0+\frac12\delta_L}(f)\\
&\qquad +\frac{2\omega L}{(2+\omega L)^2}\left(\frac{1}{L}\int_0^Lf(x)\diff x-\frac12f(0)-\frac{1}{2}f(L)\right)^2\\
&= \frac{\omega L}{2+\omega L}\mbox{Var}_{\mathcal U_{[0,L]}}(f)+ \frac{1/2}{2+\omega L}\left(\int_0^L f'(x)\diff x\right)^2\\
&\qquad+\frac{2\omega L}{(2+\omega L)^2}\left(\int_0^L f'(x)\frac{x-L/2}{L} \diff x\right)^2.
\end{align*}
We apply the Poincaré inequality for the uniform measure and Cauchy-Schwarz to get
$$
\mbox{Var}_\mu (f)\le \left(\frac{L}{\pi^2}\frac{\omega L}{2+\omega L}+\frac{L}{2(2+\omega L)}+\frac{2\omega L}{(2+\omega L)^2} \cdot \frac{{L}}{12}\right)\int_0^L (f'(x))^2\diff x
$$
using $\int_0^{L} (\frac{x - {L}/2}{{L}})^2 \diff x = \frac{{L}}{12}$. The claim follows since
\begin{align*}
    \MoveEqLeft\left(\frac{L}{\pi^2}\frac{\omega L}{2+\omega L}+\frac{L}{2(2+\omega L)}+\frac{2\omega L}{(2+\omega L)^2} \cdot \frac{{L}}{12}\right) \frac{2 + \omega L}{\omega}=\frac{L^2}{\pi^2}+\frac{L}{2\omega}+\frac{L^2}{6(2+\omega L)}.\qedhere
\end{align*}

\end{proof}

Let us now prove that the generator of sticky Brownian motion has a discrete spectrum. While this result is surely known, we were unable to find a suitable reference.
\begin{Lem}\label{lem:stickydiscrete}
    The generator $(\mL,\dom(\mL))$ of sticky Brownian motion has discrete spectrum.
\end{Lem}
\begin{proof}
    We will prove a Nash inequality for the Dirichlet form associated to the sticky Brownian motion and invariant measure $\mu$. Starting from \cite[Chap.~8 Eq.~(42)]{bre11}, a Nash inequality holds for any interval in dimension 1 for the Lebesgue measure, i.e.~there exists $c > 0$ such that
    $$ \left(\int_0^L f^2\diff x\right)^{1/2}\le c \left(\int_0^L |f|\diff x\right)^{2/3}\left(\left(\int_0^Lf^2\diff x\right)^{1/2}+\left(\int_0^L (f')^2\diff x\right)^{1/2}\right)^{1/3}.$$
    By \cite{wan00} or \cite[Chap.~8 Eq.~(43)]{bre11}, this Nash inequality is equivalent to a super Poincaré inequality: For all $s>0$
    $$\int_0^L f^2\diff x\le s\int_0^L (f')^2\diff x+\beta(s)\left(\int_0^L|f|\diff x\right)^2,$$
    with $\beta(s) $ going to infinity of order $\sim s^{-2}$ when $s$ goes to 0, so that we may assume $\beta(s)>1$. Now, let us prove that it implies the same type of super Poincaré inequality for the Dirichlet form associated to sticky Brownian motion. Denoting $a_0=\frac1{2+\omega L}$ and $b_0=\omega a_0$, for all $s>0$ we have
    \begin{align*}
        \int f^2\diff\mu &=a_0(f(0)^2+f(L)^2)+b_0\int_0^L f^2\diff x,\\
                     &\le a_0(f(0)^2+f(L)^2)+b_0 s\int_0^L (f')^2\diff x+b_0\beta(s)\left(\int_0^L|f|\diff x\right)^2,\\
                     &\le b_0 s\int_0^L (f')^2\diff x +b_0\max(b_0^{-2},a_0^{-1})\beta(s)\left(\int|f|\diff\mu\right)^2.
    \end{align*}
    Therefore, the sticky Brownian satisfies a super Poincaré inequality. Then by \cite[Th.~5.1]{wan00}, it has an empty essential spectrum. Now, by \cite[Th.~A.6.4]{BGL14}, the resolvent is compact and thus the generator has discrete spectrum.
\end{proof}

\begin{corollary}
    There exists a universal constant $K > 0$ s.t.~for all $\omega, L > 0$ there exists $C = C(\omega, L) > 0$ such that the transition semigroup $(\hat P_t)$ of the RTP process on $[0, L]$ with parameter $\omega$ satisfies
    $$
        \Vert \hat P_t f \Vert_{L^2(\hat\mu)} \le C e^{-\frac{K \omega}{1 + (\omega L)^2} t} \Vert f \Vert_{L^2(\hat\mu)}
    $$
for all $t \ge 0$ and $f \in L^2_0(\hat\mu)$. \end{corollary}

Note that the relaxation time corresponding to this decay rate is of the same order as the mixing time obtained in~\cite{GHM24}. It reveals the existence of two regimes controlled by the parameter $\omega L$. In the ballistic regime $\omega L \ll 1$, velocity flips are rare, leading to a fast exploration of the position space $\S$ and a comparatively slow exploration of the velocity space $\V$. This results in the scaling $\nu \propto \omega$. On the contrary, in the diffusive regime $\omega L \gg 1$, the high frequency of velocity flips makes the exploration of $\V$ faster than the exploration of $\S$. This leads to the scaling $\nu \propto \omega^{-1} L^{-2}$.

\begin{proof}
    It suffices to check Assumptions~\ref{asm:A}, \ref{asm:B} and~\ref{asm:C} and apply Corollary~\ref{coro:decaylift}.
    
    Assumption~\ref{asm:A}: Recall that $\dom(\mL_{C^0})$ is a core of $\mL$ by Theorem~\ref{thm:RTP_collapse}. For all $f \in \dom(\mL_{C^0})$ we have $\Lv (f \circ \pi) = 0$ hence $\Ld$ is a lift of $\mL$ by Remark~\ref{rem:assumptions}\eqref{rem:Ld_is_also_a_lift}. Furthermore, for $f \in \dom(\mL_{C^0})$ one has
    $$
    \Ld^*(f \circ \pi)(x,v) = -v 1_{\{0 < x < L\}} f'(x) = -\Ld (f \circ \pi)(x,v).
    $$

    Assumption~\ref{asm:B}: Follows from Lemmas~\ref{lem:stickypoincare} and~\ref{lem:stickydiscrete}.

    Assumptions~\ref{asm:C} (ii): A straightforward computation yields
    $$
    \int_\V \Lv f(x, v) \diff\kappa_x(v) = 0 \text{ for all } x \in \S \text{ and } f \in \dom(\Lhat).
    $$
    Let us we prove $\int_0^T \Vert P_t f - \Pi_v P_t f \Vert_{L^2(\hat\mu)}^2 \diff t \le \LvPoincareConstant \int_0^T \mathcal E_v(P_t f) \diff t$ with $m_v=2$. Define the matrices 
    \begin{align*}
    S = \begin{pmatrix}
    1/4 & 0 & 0 \\
    0 & 1/2 & 0 \\
    0 & 0 & 1/4
    \end{pmatrix},
    \qquad
    \mathcal Q = 
    \begin{pmatrix}
    -2 & 2 & 0 \\
    1 & -2 & 1 \\
    0 & 2 & -2
    \end{pmatrix},
    \end{align*}
    as well as the scalar product $\left<x, y\right>_S = x^\top S y$ and let $\Pi$ be the orthogonal projection on the kernel of $\mathcal Q$ with respect to~$\left< \cdot, \cdot \right>_S$. The matrix $\mathcal Q$ is symmetric w.r.t.~the scalar product $\left< \cdot, \cdot \right>_S$ (i.e.~$\mathcal Q = S^{-1} \mathcal Q^\top S$) so it admits an orthonormal basis of eigenvectors. Its eigenvalues are $0, -2, -4$ and thus
    \begin{equation*}\label{eq:bilinear_form_comparison}
        \Vert \xi - \Pi \xi \Vert_S^2 \le \frac12 \left<\xi, (-\mathcal Q) \xi\right>_S\quad\text{for all } \xi \in \mathbb R^3.     
    \end{equation*}
    For $f \in \dom(\Lhat)$, taking $\xi = (P_t f(x, 2), P_t f(x, 0), P_t f(x, -2))^\top$ in the inequality above yields
    \begin{equation}\label{eq:mv_assu_RTP}
        \kappa_x(P_tf(x, \cdot)^2) - \kappa_x(P_tf(x, \cdot))^2 \le \frac12 \int_\V P_tf(x, v) (-\Lv P_tf)(x, v) \diff\kappa_x(v)
    \end{equation}
    for $x \in (0, L)$. A similar argument leads to the same inequality for $x = 0$ and $x = L$. Hence
    \begin{align*}
        \int_0^T \Vert P_t f - \Pi_v P_t f \Vert_{L^2(\hat\mu)}^2 \diff t &= \int_0^T\int_{[0, L]} \left[ \kappa_x(P_t f(x, \cdot)^2) - \kappa_x(P_t f(x, \cdot))^2 \right] \diff\mu(x) \diff t \\
        &\le \frac12 \int_0^T\int_{[0, L]} \int_\V P_t f(x, v) (-\Lv P_t f)(x, v) \diff\kappa_x(v) \diff\mu(x) \diff t.
    \end{align*}
    Therefore, \eqref{eq:mvassu} holds with $\LvInvPoincareConstant=2$.

    Assumption~\ref{asm:C} (i): Let $f\in\dom(\Ld)$ and $g \in \dom(\mL_{C^0})$. We have that $u = f - \Pi_v f$ satisfies $\Pi_v u = 0$, so, in particular,
    \begin{align*}
        u(x, 2) + u(x, -2) = -2 u(x, 0)\qquad\text{for all }x\in(0,L).
    \end{align*}
    Therefore
    \begin{align*}
        &\big< \Ld (g  \circ \pi), \Ld (f - \Pi_v f)\big>_{L^2(\hat\mu)} \\
        	&\qquad= \frac{\omega/4}{2 + \omega L}\int_0^L (+2)\partial_x g(x) (+2) \partial_x u(x, +2) \diff x   \\
    &\qquad\qquad\qquad+	\frac{\omega/4}{2 + \omega L}\int_0^L (-2)\partial_x g(x) (-2) \partial_x u(x, -2) \diff x  \\
    &\qquad=	\frac{\omega}{2 + \omega L}\int_0^L \partial_x g(x) (\partial_x u(x, +2) + \partial_x u(x, -2)) \diff x  \\
    &\qquad=
    	\frac{-2\omega}{2+\omega L} \int_0^L \partial_x g(x) \partial_x u(x, 0) \diff t,\\
    	&\big|\big< \Ld (g  \circ \pi), \Ld (f - \Pi_v f)\big>_{L^2(\hat\mu)} \big|\\
        &\qquad = \big|\big< 2 \mathcal L g, u(\cdot, 0)\big>_{L^2(\mu)} \big| \ \le\ \Vert 2 \mathcal L g\Vert_{L^2(\mu)} \Vert u(\cdot, 0) \Vert_{L^2(\mu)}.
    \end{align*}
    Finally,
    \begin{align*}
    	\Vert u( \cdot, 0) \Vert_{L^2(\mu)}^2 &= \frac1{2+\omega L} u( 0, 0)^2 + \frac\omega{2 + \omega L}\int_0^L u( x, 0)^2 \diff x+ \frac1{2+\omega L} u( L, 0)^2 \\
    	&\le 2\Bigg[\frac{1}{2 + \omega L} \left( \frac12 u( 0, -2)^2 + \frac12 u( 0, 0)^2\right) \\
    	&\quad\quad\quad+ \frac\omega{2+\omega L} \int_0^L \left( \frac14 u( x, -2)^2 + \frac12 u( x, 0)^2 + \frac14 u( x, +2)^2\right) \diff x \\
    	&\quad\quad\quad+ \frac{1}{2 + \omega L} \left( \frac12 u( L, +2)^2 + \frac12 u( L, 0)^2\right) \Bigg] \\
    	&\le 2 \Vert u \Vert_{L^2(\hat\mu)}^2,
    \end{align*}
    and hence
    \begin{align*}
        \big|\big< \Ld (g  \circ \pi), \Ld (f - \Pi_v f)\big>_{L^2(\hat\mu)}\big| \le \sqrt{2} \Vert 2 \mathcal L g\Vert_{L^2(\mu)} \Vert f - \Pi_v f \Vert_{L^2(\hat\mu)} \le \Vert 2 \mathcal L g\Vert_{L^2(\mu)} \sqrt{\mathcal E_v(f)}
    \end{align*}
    using~\eqref{eq:mv_assu_RTP} for the last inequality. Therefore, \eqref{eq:C1} is satisfied with $C_1 = 1$. The arguments used to check~\eqref{eq:mvassu} show that $\Lv$ is a bounded operator and hence~\eqref{eq:C2} is satisfied with $C_2 = O(m^{-1/2})$ by Remark~\ref{rem:assumptions}\eqref{rem:assumption_C_full_refresh}.

    Remark~\ref{rem:decay_rate_in_simple_setting} thus yields the desired decay rate.
\end{proof}

\subsection{Proof of Theorem~\ref{thm:invariant_measure_generator_RTP}}\label{ssec:RTPdomainproof}

The following lemma is surely known, or at least folklore, but we did not find a reference, so the proof is included.
\begin{Lem}\label{lem:comparing_Feller_and_extended_generator}
    Let $(\mathcal L_\mathrm{ext}, \dom(\mathcal L_\mathrm{ext}))$ (resp.~$(\mathcal L_{C^0}, \dom(\mathcal L_{C^0}))$) be the extended (resp.~$C^0$) generator of a Feller process $X(t)$ taking its values in the compact state space $\S$.
    Then $\dom(\mathcal L_{C^0}) = \{ f \in C^0(\S)\colon f \in \dom(\mL_\mathrm{ext}) \text{ and } \mL_\mathrm{ext} f \in C^0(\S) \}$ and
    $$
        \mL_{C^0} f = \mL_\mathrm{ext} f\qquad\text{for all } f \in \dom(\mathcal L_{C^0}).
    $$
\end{Lem}

\begin{proof}
    It follows from~\cite[Theorem 3.32]{liggett10} that the extended generator is an extension of the $C^0$-generator.

    Let $f \in C^0(\S)$ be such that $f \in \dom(\mL_\mathrm{ext})$ and $\mL_\mathrm{ext} f \in C^0(\S)$. Then
    $$
        M(t) = f(X(t)) - f(X(0)) - \int_0^t \mL_\mathrm{ext}f(X(s)) \diff s
    $$
    is a local martingale. Hence
    $$
        \mathbb{E} \left[ \sup_{0 \le s \le t} \left|f(X(t)) - f(X(0)) - \int_0^t \mL_\mathrm{ext}f(X(s)) \diff s \right| \right] \le 2 \Vert f \Vert_\infty + t \Vert \mathcal L_\mathrm{ext} f \Vert_\infty < +\infty
    $$
    implies that $M(t)$ is a martingale and thus
    \begin{align*}
        P_t f(x) - f(x) = \mathbb E_x \left[ f(X(t)) - f(X(0)) \right] = \mathbb E_x \left[ \int_0^t \mL_\mathrm{ext} f (X(s)) \diff s \right] = \int_0^t \mathbb (P_s \mL_\mathrm{ext} f)(x) \diff s.
    \end{align*}
    The function $s \mapsto P_s \mL_\mathrm{ext} f$ is continuous from $\mathbb R_+$ to $C^0(\S)$ hence the integral $\int_0^t P_s \mL_\mathrm{ext} f \diff s$ is well-defined in $C^0$ and the pointwise equality above implies
    $$
        \frac{P_t f - f}{t} = \frac1t \int_0^t P_s \mL_\mathrm{ext} f \diff s \to \mL_\mathrm{ext} f\qquad\text{as }t\to 0,
    $$
    using the continuity of $s \mapsto P_s \mL_\mathrm{ext} f$ to obtain a $\Vert \cdot \Vert_\infty$-limit.
\end{proof}

\begin{Lem}[$C^0$-generator of the RTP process]\mbox{}
\begin{itemize}
    \item[(i)] The RTP process is Feller.
    \item[(ii)] Its $C^0$-generator is given by
    \begin{align*}
    	\Lhat_{C^0} f(x, 2) &= 2 \partial_x f(x, 2) + 2\omega \left( f(x, 0) - f(x, 2) \right), \\
    	\Lhat_{C^0} f(x, 0) &= \omega f(x, 2) + \omega f(x, -2) - 2\omega f(x, 0), \\
    	\Lhat_{C^0} f(x, -2) &= -2 \partial_x f(x, -2) + 2\omega \left( f(x, 0) - f(x, -2) \right),
    \end{align*}
    and its domain $\dom(\Lhat_{C^0})$ consists of all \( f \in C^0 ( \Shat) \) such that
\begin{align*}
	f(\cdot, \pm2) \in C^1([0, L]) \quad \text{ and } \quad \partial_x f(L, 2) =\partial_x f(0, -2) = 0.
\end{align*}
\end{itemize}
\end{Lem}

\begin{proof}
    (i) It is enough to show that $\mathbb P_{(\tilde x_0, v_0)}$ converges weakly to $\mathbb P_{( x_0,  v_0)}$ when $\tilde x_0 \to x_0$. Following arguments in \cite{GHM24}, we consider the coupling of $\mathbb P_{(\tilde x_0, v_0)}$ and $\mathbb P_{( x_0,  v_0)}$ obtained by taking the same relative velocity process $v(t)$. Since $y \mapsto \min({L}, \max(0, y + v(0) t))$ is 1-Lipschitz, we have
\begin{align*}
\left| x(t) - \tilde x(t)\right|=\left|\min({L}, \max(0, x_0 + v(0) t)) - \min({L}, \max(0, \tilde x_0 + v(0) t))\right| \le \left| x_0 - \tilde x_0 \right|
\end{align*}
for $t \le T_1$. In particular, $\left|x(T_1) - \tilde x(T_1)\right| \le \left| x_0 -\tilde x_0\right|$. So, by induction,
$
\left| x(t) - \tilde x(t) \right| \le \left| x_0  - \tilde x_0 \right|${ for all }$t \ge 0$,
and the Feller property follows.

(ii) The claim follows from the derivation of the extended generator in \cite{GHM24} and Lemma~\ref{lem:comparing_Feller_and_extended_generator}.
\end{proof}

The invariance of $\hat\mu$ (Theorem~\ref{thm:invariant_measure_generator_RTP} (i)) is already proven in \cite{hahn23} and can be directly checked here using the generator~\cite[Theorem 3.37]{liggett10}. We now provide the proof for the \( L^2(\hat\mu) \) generator and its domain.

\begin{proof}[Proof of Theorem~\ref{thm:invariant_measure_generator_RTP} (ii)]  The \( L^2(\hat\mu) \) generator is the closure of \( (\Lhat_{C^0}, \dom(\Lhat_{C^0})) \) in \( L^2(\hat\mu) \) by~\cite[Proposition II.1.7]{engel99}. We concentrate on the case \( v = 2 \) as cases \( v = 0 \) and \( v = -2 \) follow from the same arguments.
Denote $(\Lhat, \dom(\Lhat))$ the operator defined in Theorem \ref{thm:invariant_measure_generator_RTP} (ii).
 
We first show \( \overline{(\Lhat_{C^0}, \dom(\Lhat_{C^0}))} \subseteq (\Lhat, \dom(\Lhat)) \): Let \( f_n \in \dom(\Lhat_{C^0}) \) be such that
$$
 f_n \longrightarrow f \quad \text{ and } \quad \Lhat_{C^0} f_n \longrightarrow g\qquad\text{in }L^2(\hat\mu).
$$
We need to show that $f\in\dom(\Lhat)$ and $\Lhat f=g$. On $(0, L)$ one has
\[
	\partial_x f_n(\cdot, 2) = \frac12 \left( \Lhat_{C^0} f_n(\cdot, 2) - 2\omega \left[f_n(\cdot, 0) - f_n(\cdot, 2)\right] \right)
\]
hence \( \partial_x f_n(\cdot, 2) \) converges in \( L^2((0, L),\diff x) \) to \( h = \frac12\left(g(\cdot, 2) - 2\omega \left(f(\cdot, 0) - f(\cdot, 2)\right) \right)\).
For \( \phi \in C^\infty_c((0, L)) \) we have
\begin{align*}
	-\int_0^L f(\cdot, 2) \phi' \diff x = -\lim_{n \to \infty}\int_0^L f_n(\cdot, 2) \phi' \diff x = \lim_{n \to \infty} \int_0^L \partial_x f_n(\cdot, 2) \phi \diff x = \int_0^L h \phi \diff x.
\end{align*}
Hence \( f(\cdot, 2) \in H^1((0, L)) \), \( f_n (\cdot, 2) \to f(\cdot, 2) \) in \( H^1((0, L)) \), and \( g(x, 2) = 2 \partial_x f(x, 2) + 2\omega \left( f(x, 0) - f(x, 2) \right) \) for \( x \in (0, L) \).
Since function evaluations at $(L,2)$ and $(L,0)$ are well-defined and continuous from \( L^2(\hat\mu) \) to \( \mathbb R \), we obtain \( g(L, 2) = \omega \left( f(L, 0) - f(L, 2) \right) \).
Furthermore, \( F \mapsto F(L-, 2) \) is the trace operator which is continuous on \( H^1((0, L)) \), so that
\[
	f_n(L-, 2) = f_n(L, 2) \text{ for all } n \ge 0 \implies f(L-, 2) = f(L, 2),
\]
which yields $f\in\dom(\Lhat)$ and $\Lhat f = g$.

It remains to show \( (\Lhat, \dom(\Lhat)) \subseteq \overline{(\Lhat_{C^0}, \dom(\Lhat_{C^0}))} \): Let \( f \in \dom(\Lhat) \) be given. We now construct a sequence \( f_n \in \dom(\Lhat_{C^0}) \) such that \( f_n \longrightarrow f \) and \( \Lhat_{C^0} f_n \longrightarrow \Lhat f \) in $L^2(\hat\mu)$.

Let \( h_n \in C^\infty((0, L)) \) be a sequence converging to \( f(\cdot, 2) \) in \( H^1((0, L)) \).
Define
\[
	\chi_n(x) =
	\left\{
	\begin{array}{cl}
		0 & \text{ if } x \le L(1 - 1/n), \\
		\frac{x - L(1 - 1/n)}{L(1-1/(2n)) - L(1 - 1/n)} & \text{ if } x \in (L(1 - 1/n), L(1-1/(2n))), \\
		1 & \text{ if } x \ge L(1 - 1/(2n)).
	\end{array}
	\right.
\]
which is in \( C^0([0, L]) \).
Set \( g_n(x) = h_n(x) - \int_0^x h'_n(y) \chi_n(y) \diff y\). One has \( g_n \in C^1([0, L]) \) as well as
\[
	g'_n(L) = h'_n(L) - h'_n(L) \chi_n(L) = 0.
\]
Furthermore,
\begin{align*}
	\left(\int_0^x h'_n \chi_n \diff y \right)^2 &\le \left(\int_0^L (h'_n)^2 \diff y \right) \left( \int_0^L \chi_n^2 \diff y \right) \le \left(\int_0^L (h'_n)^2 \diff y \right) \cdot \frac Ln
\end{align*}
so that \( x\mapsto \int_0^x h'_n(y) \chi_n(y) \diff y \) converges to \( 0 \) uniformly and thus in \( L^2((0, L),\diff x) \).
For all \( a \in (0, L) \) we have
\begin{align*}
	\lim_{n\to\infty} \int_0^L (h_n' \chi_n)^2 \diff y \le \lim_{n\to\infty} \int_{a}^L (h_n')^2 \diff y = \int_{a}^L (\partial_x f(\cdot, +2))^2 \diff y,
\end{align*}
so taking \( a \to L \) yields
$$
	\lim_{n\to\infty} \int_0^L (h_n' \chi_n)^2 \diff y = 0.
$$
Hence \( x\mapsto \int_0^x h'_n(y) \chi_n(y) \diff y \) converges to \( 0 \) in \( H^1((0, L)) \). We deduce that \( g_n \) converges to \( f(\cdot, 2) \) in \( H^1((0, L)) \).
Because the trace \( F \mapsto F(L-, 2) \) is continuous from \( H^1((0, L)) \) to \( \mathbb R \) and \( f(L-, 2) = f(L, 2) \) we have that \( g_n \to f(\cdot, 2) \) in \( L^2([0,L],\lambda_{[0,L]} + \delta_L) \) as required.
Thus taking $f_n(x, +2) = g_n(x)$ yields convergence of \( (\Lhat_{C^0} f_n)(\cdot, +2) \) to \( \Lhat f(\cdot, +2) \) in \( L^2(\lambda_{[0,L]} + \delta_L) \). The cases $v = 0$ and $v = -2$ can be treated similarly to obtain an approximating sequence for $f\in \dom(\hat L)$ that converges in $L^2(\hat\mu)$.
\end{proof}

\section{PDMPs for sampling}\label{sec:pdmps}

In this section, we consider a probability measure $\mu$ on $\R^d$ with density proportional to $\exp(-U)$, where $U\colon\R^d\to\R$ is such that $\int_{\R^d}e^{-U(x)}\diff x<\infty$. We assume that $\mu$ satisfies a Poincar\'e inequality with constant $\frac{1}{m}$ and that the generator $(\mL,\dom(\mL))$ of the overdamped Langevin diffusion with invariant probability measure $\mu$, see Example~\ref{ex:langevinlift}, has discrete spectrum. The latter is for instance satisfied if there exists $\beta>1$ such that
\begin{equation*}
    \liminf_{|x|\to\infty} \frac{\nabla U(x)\cdot x}{|x|^\beta} >0,
\end{equation*}
see \cite[Theorem 5]{BGW15}. As in Example~\ref{ex:langevin}, we assume the lower curvature bound
\begin{equation}\label{eq:curvaturebound2}
    \nabla^2U(x)\geq -Km \cdot I_d\qquad\text{for all }x\in\R^d
\end{equation}
for some $K\geq 0$.

Piecewise deterministic Markov processes (PDMPs) have recently been studied as highly efficient non-reversible samplers in physics and statistics, where proving exponential ergodicity is crucial. Depending on communities or the specific choice of dynamics, these processes appear under different algorithmic names, including (Forward) Event Chain Monte Carlo \cite{MKK14,MDS20}, the Zig-Zag sampler \cite{BRZ19} and the Bouncy Particle sampler \cite{durmus20}.
Very recently, using the Dolbeault-Mouhot-Schmeiser approach \cite{DMS15}, results with well-controlled dimensional dependence were established in \cite{ADNR21}. These results were refined in \cite{lu22} using the variational approach by Albritton, Armstrong, Mourrat and Novack for the Zig-Zag process and the Bouncy Particle sampler. Here, we extend these results using our method, treating in particular two previously unaddressed cases: the true Zig-Zag process, without the condition of Gaussian velocities, and the Forward generalisation that incorporates stochastic jumps at events.

\subsection{Forward process}

The forward Event Chain process introduced in \cite{MDS20} is a PDMP $(X_t,V_t)$ with state space $\Shat = \R^d\times\R^d$ and invariant probability measure $\hat\mu = \mu\otimes\kappa$, where $\kappa$ is the $d$-dimensional standard Gaussian. It extends previously developed algorithms, as the Bouncy Particle Sampler and the Zig-Zag process, as the velocity update at an event is now governed more generally by a Markov kernel. We consider the \emph{full-orthogonal refresh} variant whose generator is given by 
\begin{align*}
    \Lhat f(x, v) = v \cdot \nabla_x f(x,v) +(v\cdot\nabla U(x))_+(B_-f -f)(x,v) + \gamma (\Pi_vf-f)(x,v)
\end{align*}
for $f\in C_\comp^1(\R^{2d})$ and $(x,v)\in\R^d\times\R^d$, where
\begin{equation*}
    B_\pm f(x,v)=\frac1{Z_x}\int_{\R^d}(w\cdot\nabla U(x))_\pm f(x,w)\kappa(\diff w)
\end{equation*}
and $Z_x = \int_{\R^d}(w\cdot\nabla U(x))_+\kappa(\diff w) = \frac{1}{\sqrt{2\pi}}|\nabla U(x)|$.
It moves along straight lines between events occurring with rate $(v\cdot\nabla U)_+$. At an event, the velocity $V_t$ is resampled according to the distribution $Z_{X_t}^{-1}(v\cdot\nabla U(X_t))_-\kappa(\diff v)$, corresponding to the distribution for the reflection of the new direction to trigger an event. This corresponds to a resampling of the whole vector $V_t$, including its component along the direction $\nabla U(X_t)$, as opposed to a simple reflection or flip as respectively for the Bouncy Particle Sampler or Zig-Zag process. 
The forward Event Chain process admits a natural splitting of its generator. For $f\in C_\comp^1(\R^{2d})$, we set
\begin{align*}
    \Ld f(x,v) &= v\cdot\nabla_xf(x,v)+(v\cdot\nabla U (x))_+(B_-f-f)(x,v)\quad\text{and}\\
    \Lv f(x,v) &= \int_{\R^d} f(x,w)\kappa(\diff w)-f(x,v) = \Pi_vf(x,v)-f(x,v),
\end{align*}
so that $\Lhat f = \Ld f + \gamma\Lv f$.

\begin{asm}\label{asm:U2}
    There exist $a \in [0, 1)$ and $b \ge 0$ such that
    \begin{equation}\label{eq:assumption_on_U}
        \Delta U(x) \le a |\nabla U(x)|^2 + b\qquad\text{for all }x\in\R^d.
    \end{equation}
\end{asm}
For example, Assumption~\ref{asm:U2} is satisfied with $a=0$ and $b=dL$ if $\nabla^2U(x)\leq L\cdot I_d$ for all $x\in\R^d$. 

\begin{corollary}\label{coro:rateforward} 
    Let Assumption~\ref{asm:U2} be satisfied. Then, choosing $T=m^{-1/2}$, the transition semigroup of the Forward process is exponentially contractive in $T$-average with rate
    $$
        \nu = \Omega\left( \frac{\gamma m}{\gamma^2 + \frac{m(1+K)+b}{(1-a)^2}} \right).
    $$
    In particular, if $\gamma = \Theta\left(\frac{\sqrt{m(1+K)+b}}{1-a}\right)$ then
    $
        \nu = \Omega\left( \sqrt{m}(1-a)\left/\sqrt{1+K+ b/m} \right.\right).
    $
\end{corollary}

\begin{Rem}
    In case $m\cdot I_d\leq\nabla^2U\leq L\cdot I_d$, and with $\gamma = \Theta\left(\sqrt{m+dL}\right)$, we obtain
    $
        \nu = \Omega\left(\sqrt{m}\big/\sqrt{d{L}/{m})}\big.\right) 
    $.
    We thus recover the same dimension dependence as for the Bouncy Particle sampler in~\cite{lu22} but an additional dependence on the condition number.

\end{Rem}

\begin{proof}[Proof of Corollary~\ref{coro:rateforward}]
    Assumptions~\ref{asm:A} and \ref{asm:B} can be verified as for the Langevin dynamics in Example~\ref{ex:langevin}. Indeed, the forward Event Chain process is a lift of an overdamped diffusion with generator $(\mL,\dom(\mL))$, and Assumption~\ref{asm:C}(ii) is satisfied with $\LvInvPoincareConstant=1$. Assumption \eqref{eq:C2} holds with $C_2=\frac{1}{\sqrt{2m}}$ by Remark~\ref{rem:assumptions}\eqref{rem:assumption_C_full_refresh}, we hence focus on Assumption~\eqref{eq:C1}.

    It is enough to estimate $\Vert\Ld^*\Ld(g \circ \pi) + (\mL g) \circ \pi\Vert_{L^2(\hat\mu)}$ by Remark~\ref{rem:assumptions}\eqref{rem:C1_with_transpose}. Noticing
    \begin{align*}
        \Ld(g\circ\pi)=v\cdot\nabla_xg\quad\text{and}\quad \Ld^*f=-v\cdot\nabla_xf(x,v)+(v\cdot\nabla U(x))_-(B_+f-f)(x,v)
    \end{align*}
    for $f\in C_\comp^1(\R^{2d})$,
    one has
    \begin{align*}
    \Ld^*\Ld(g \circ \pi)(x,v)=-v^\top\nabla^2 g(x)v+\frac1{Z_x}(v\cdot\nabla U(x))_-\left(\frac12\nabla U(x)\cdot\nabla g(x)-Z_xv\cdot\nabla g(x)\right).
    \end{align*}
    By Cauchy-Schwarz,
    \begin{align*}
        \int (\Ld^*(v\cdot\nabla_xg))^2\kappa(\diff v)&\leq 2\int (v\cdot\nabla^2gv)^2\kappa(\diff v)\\
        &\quad+ \frac{2}{Z^2}\int(v\cdot\nabla U)_-^2\left(\frac12\nabla U\cdot\nabla g-Zv\cdot\nabla g\right)^2\kappa(\diff v)\\
        &=4\|\nabla^2g\|_F^2 + 2(\Delta g)^2\\
        &\quad +2(3+\frac{\pi}{4})(\nabla U\cdot\nabla g)^2 + |\nabla U|^2|\nabla g|^2.
    \end{align*}
    where $\Vert \cdot\Vert_F$ is the Frobenius norm. Using \eqref{eq:assumption_on_U} we estimate
    \begin{align*}
        \int_{\R^d}|\nabla U|^2|\nabla g|^2\diff\mu &= \int_{\R^d}\Delta U|\nabla g|^2\diff\mu + 2\int_{\R^d}\nabla U^\top\nabla^2g\nabla g\diff\mu\\
        &\leq a\int_{\R^d}|\nabla U|^2|\nabla g|^2\diff\mu + b\int_{\R^d}|\nabla g|^2\diff\mu + 2\int_{\R^d}\nabla U^\top\nabla^2g\nabla g\diff\mu.
    \end{align*}
    Now using $2\nabla U^\top\nabla ^2 g\nabla g\leq 2|\nabla U||\nabla g|\lVert\nabla^2g\rVert_F\leq p|\nabla U|^2|\nabla g|^2 + \frac{1}{p}\lVert\nabla ^2g\rVert_F^2$ for any $p>0$, choosing $p=\frac{1-a}{2}$ and rearranging yields
    \begin{equation}\label{eq:nablaUnablagbound}
        \begin{aligned}
        \int_{\R^d}|\nabla U|^2|\nabla g|^2\diff\mu &\leq \frac{2b}{1-a}\int_{\R^d}|\nabla g|^2\diff\mu + \frac{4}{(1-a)^2}\int_{\R^d}\lVert\nabla^2g\rVert_F^2\diff\mu\\
        &\leq \left(\frac{2b}{1-a}+\frac{4Km}{(1-a)^2}\right)\int_{\R^d}|\nabla g|^2\diff\mu + \frac{4}{(1-a)^2}\lVert 2\mL g\rVert^2_{L^2(\mu)}.
    \end{aligned}
    \end{equation}
    Finally, the bounds $(\Delta g)^2\leq 2(2\mL g)^2 + (\nabla U\cdot\nabla g)^2$ and $(\nabla U\cdot\nabla g)^2\leq |\nabla U|^2|\nabla g|^2$ as well as Bochner's identity yield
    \begin{align*}
        \MoveEqLeft\Vert \Ld^*\Ld (g \circ \pi) \Vert_{L^2(\hat\mu)}^2
        \\
        &\le C\left( \big\Vert \Vert \nabla^2g\Vert_F \big\Vert_{L^2(\mu)}^2 + \Vert \Delta g\Vert_{L^2(\mu)}^2 + \Vert \nabla U\cdot\nabla g\Vert_{L^2(\mu)}^2 + \big\Vert |\nabla U||\nabla g|\big\Vert_{L^2(\mu)}^2 \right) \\
        &\leq C\left(\lVert2\mL g\rVert_{L^2(\mu)}^2 + 2Km\E(g) + \lVert|\nabla U||\nabla g|\rVert_{L^2(\mu)}^2\right)\\
        &\leq C\left(\frac{1}{(1-a)^2}\Vert 2\mL g\Vert_{L^2(\mu)}^2 + \frac{b(1-a)+Km}{(1-a)^2}(2\E(g))\right)\\
        &\leq C\left(\frac{b(1-a)+(1+K)m}{m(1-a)^2}\right)\lVert2\mL g\rVert_{L^2(\mu)}^2
    \end{align*}
    where $C>0$ denotes an absolute constant changing from line to line.
    Therefore, overall, one can choose $C_1^2 = C\Big(\frac{{b+(1+K)m}}{m(1-a)^2}\Big)$ and the claim follows by Remark~\ref{rem:decay_rate_in_simple_setting}.
\end{proof}

\subsection{Zig-Zag process}

The convergence to equilibrium for the Zig-Zag process with Gaussian velocities is considered in \cite{lu22}. In practice, one usually considers velocities chosen uniformly among the (scaled) coordinate directions $\{\pm\sqrt{d}e_i\colon i=1,\dots,n\}$ or from the hypercube $\{-1,+1\}^d$ \cite{BRZ19}. We consider any of these dynamics, i.e.\ let $\V$ be $\R^d$, $\{\pm\sqrt{d}e_i\colon i=1,\dots,n\}$, or $\{-1,+1\}^d$, and consider $\kappa = \mathcal{N}(0,I_d)$ or $\kappa = \mathrm{Unif}(\V)$, correspondingly.

The Zig-Zag process $(X_t,V_t)$ is a PDMP with state space $\Shat = \R^d\times\V$ that, between events, moves with constant velocity $v\in\V$. The $k$-th component $v_k$ of the velocity is then flipped with rate $(v_k\partial_kU(x))_+$. Additionally, with rate $\gamma>0$, a new velocity is chosen from $\V$ according to the distribution $\kappa$. Note that this velocity refreshment mechanism may not be needed for ergodicity in the hypercube case~\cite{BRZ19}, but it is essential in our approach. This process has $\hat\mu = \mu\otimes\kappa$ as its invariant probability measure \cite{BRZ19}, and the generator of the associated transition semigroup in $L^2(\hat\mu)$ is given by
\begin{align*}
    \Lhat f(x,v)&=v\cdot\nabla_xf(x,v)+\sum_{i=1}^d(v_i\partial_i U(x))_+(f(x,v-2e_iv_i)-f(x,v))\\
    &\qquad+\gamma \left(\int f(x,w)\kappa(\diff w)-f(x,v)\right)
\end{align*}
for compactly supported functions $f\in C_\comp^1(\R^{2d})$ and $(x,v)\in\R^d\times\V$. Under additional assumptions, see \cite{Durmus2021PDMP}, $C_\comp^1(\R^d)$ forms a core for $\Lhat$. For $f\in\dom(\Lhat)$, we set
\begin{align*}
    \Ld f(x,v) &= v\cdot\nabla_xf(x,v)+\sum_{i=1}^d(v_i\partial_i U(x))_+(f(x,v-2e_iv_i)-f(x,v))\quad\text{and}\\
    \Lv f(x,v) &= \int_\V f(x,w)\kappa(\diff w)-f(x,v) = (\Pi_vf)(x,v)-f(x,v),
\end{align*}
so that $\Lhat f = \Ld f + \gamma\Lv f$.

\begin{asm}\label{asm:U3}
    There exists $L\geq 0$ such that $\partial_{kk}U(x)\leq L$ for all $k\in\{1,\dots,d\}$ and $x\in\R^d$.
\end{asm}
In particular, Assumption~\ref{asm:U3} holds if $\nabla^2U(x)\leq L\cdot I_d$ for all $x\in\R^d$. 

\begin{corollary}\label{coro:rateZZ1}
    Let Assumption \ref{asm:U3} be satisfied.
    \begin{enumerate}[(i)]
        \item Let $\kappa = \mathcal{N}(0,I_d)$ or $\kappa = \mathrm{Unif}(\{-1,+1\}^d)$. Then, choosing $T=m^{-1/2}$, the transition semigroup $(\hat P_t)$ of the Zig-Zag process is exponentially contractive in $T$-average with rate
        \begin{equation*}
            \nu = \Omega\left(\frac{\gamma m}{\gamma^2+m(1+K)+L}\right).
        \end{equation*}
         If $\gamma = \Theta\left(\sqrt{m(1+K)+L}\right)$ then
        $
            \nu = \Omega\left( \sqrt{m}\left/\sqrt{(1+K) + L/m} \right.\right).
        $
        \item Let $\kappa = \mathrm{Unif}(\{\pm\sqrt{d}e_i\colon i=1,\dots,n\})$. Then, choosing $T=m^{-1/2}$, the transition semigroup $(\hat P_t)$ of the Zig-Zag process is exponentially contractive in $T$-average with rate
        \begin{equation*}
            \nu = \Omega\left(\frac{\gamma m}{\gamma^2+md(1+K)+Ld}\right).
        \end{equation*}
         If $\gamma = \Theta\left(\sqrt{(m(1+K)+L)d}\right)$ then
        $
            \nu = \Omega\left( \sqrt{m}\left/\sqrt{((1+K)+L/m)d} \right.\right).
        $

    \end{enumerate}
\end{corollary}

In \cite{lu22}, quantitative convergence rates for the Zig-Zag process with Gaussian velocities are shown. Corollary~\ref{coro:rateZZ1} recovers the same scaling in $m$, $d$ and $L$ under slightly weaker assumptions in the practically relevant case of velocities chosen from the hypercube \cite{bierkens19}.

\begin{proof}[Proof of Corollary~\ref{coro:rateZZ1}]
    Assumptions~\ref{asm:A} and \ref{asm:B} can be verified as for the Langevin dynamics in Example~\ref{ex:langevin}. Indeed, the Zig-Zag process satisfies $\Lhat(f\circ\pi)(x,v) = v\cdot\nabla f(x)$ for all $f\in C_\comp^\infty(\R^d)$. Since the distributions $\kappa$ we consider are centred and have identity covariance matrix, the Zig-Zag process is a lift of the overdamped Langevin diffusion with invariant measure $\mu$ and generator $\mL = -\frac{1}{2}\nabla^*\nabla$ in $L^2(\mu)$ and core $C_\comp^\infty(\R^d)$ for any of these choices of velocity distribution, and Assumption \ref{asm:C}(ii) is satisfied with $m_v=1$. Assumption~\eqref{eq:C2} holds with $C_2=1$ by Remark~\ref{rem:assumptions}\eqref{rem:assumption_C_full_refresh}, we hence focus on Assumption~\eqref{eq:C1}.

    We have
    \begin{align*}
        \Ld^*f(x,v)&=-v\cdot\nabla_xf + \sum_{k=1}^d(v_k\partial_kU(x))_-(f(x,v-2e_kv_k)-f(x,v))
    \end{align*}
    for all $f\in C_\comp^1(\R^{2d})$ and thus
    \begin{equation*}
        \Ld^*\Ld(g\circ\pi) = -v\cdot\nabla^2g(x)v -2 \sum_{k=1}^d(v_k\partial_kU(x))_-v_k\partial_kg(x).
    \end{equation*}
    for all $g\in C_\comp^\infty(\R^d)$. This gives
    \begin{align*}
        \Var_\kappa(\Ld^*\Ld(g\circ\pi))\leq 2\Var_\kappa\big(v\cdot\nabla^2gv\big) + 8\Var_\kappa\left(\sum_{k=1}^d(v_k\partial_kU)_-v_k\partial_kg\right).
    \end{align*}
    This can be bounded using the moments of $\kappa$ in the following way.
    \begin{enumerate}[(i)]
        \item If $\kappa = \mathcal{N}(0,I_d)$ or $\kappa = \mathrm{Unif}(\{-1,+1\}^d)$, then 
        \begin{align*}
            \Var_\kappa\big(v\cdot\nabla^2gv\big) &= \sum_{i=1}^d\mathbb{E}_\kappa(v_i^4)(\partial_{ii}g)^2 + \sum_{i\neq j}\mathbb{E}_\kappa(v_i^2v_j^2)(2(\partial_{ij}g)^2+\partial_{ii}g\partial_{jj}g) -(\Delta g)^2\\
            &\leq2\Vert\nabla^2 g\Vert_{F}^2,\qquad\qquad\text{and}\\
            \MoveEqLeft\Var_\kappa\left(\sum_{k=1}^d(v_k\partial_kU)_-v_k\partial_kg\right) = \sum_{k=1}^d\Var_\kappa\big((v_k\partial_kU)_-v_k\partial_kg\big)\\
            &=\sum_{k=1}^d\mathbb{E}_\kappa\left((v_k\partial_kU)_-^2(v_k\partial_kg)^2\right)-\left(\mathbb{E}_\kappa\left((v_k\partial_kU)_-v_k\partial_kg\right)\right)^2\\
            &=\sum_{k=1}^d\partial_kU^2\partial_kg^2\left(\frac{1}{2}\mathbb{E}_\kappa(v_k^4)-\left(\frac{1}{2}\mathbb{E}_\kappa(v_k^2)\right)^2\right)\\
            &=C_\kappa\sum_{k=1}^d(\partial_kU)^2(\partial_kg)^2
        \end{align*}
        with $C_\kappa=\frac{5}{4}$ in the Gaussian case, and $C_\kappa=\frac{1}{4}$ in the hypercube case.
        Hence 
        \begin{equation*}
            \Vert\Ld^*\Ld(g\circ\pi)+2(\mL g)\circ\pi\Vert_{L^2(\hat\mu)}^2 \leq 4\big\Vert\Vert\nabla^2g\Vert_F\big\Vert_{L^2(\mu)}^2 + 8C_\kappa \sum_{k=1}^d\int_{\R^d}(\partial_kU)^2(\partial_kg)^2\diff\mu. 
        \end{equation*}
    
        \item If $\kappa = \mathrm{Unif}(\{\pm\sqrt{d}e_i\colon i=1,\dots,n\})$ is the uniform distribution on the scaled coordinate directions, then
        \begin{align*}
            &\Var_\kappa\big(v\cdot\nabla^2gv\big) = d\sum_{k=1}^d(\partial_{kk}g)^2 - (\Delta g)^2 \leq d\Vert\nabla^2g\Vert_F^2 ,\qquad\qquad\text{and}\\
            &\Var_\kappa\left(\sum_{k=1}^d(v_k\partial_kU)_-v_k\partial_kg\right)\\
            &\qquad=\mathbb{E}_\kappa\left(\sum_{k=1}^d(v_k\partial_kU)_-^2(v_k\partial_kg)^2\right) - \left(\mathbb{E}_\kappa\left(\sum_{k=1}^d(v_k\partial_kU)_-v_k\partial_kg\right)\right)^2\\
            &\qquad=\frac{d}{2}\sum_{k=1}^d(\partial_kU)^2(\partial_kg)^2 - \frac{1}{4}(\nabla U\cdot\nabla g)^2\leq \frac{d}{2}\sum_{k=1}^d(\partial_kU)^2(\partial_kg)^2
        \end{align*}
        This gives
        \begin{equation*}
            \Vert\Ld^*\Ld(g\circ\pi)+2(\mL g)\circ\pi\Vert_{L^2(\hat\mu)}^2 \leq 2d\big\Vert\Vert\nabla^2g\Vert_F\big\Vert_{L^2(\mu)}^2 + 4d \sum_{k=1}^d\int_{\R^d}(\partial_kU)^2(\partial_kg)^2\diff\mu. 
        \end{equation*}
    \end{enumerate}
    It thus remains to bound $\sum_{k=1}^d\int_{\R^d}(\partial_kU)^2(\partial_kg)^2\diff\mu$ in all three cases.
    An integration by parts and $2ab\leq \frac{1}{2}a^2+2b^2$ gives
    \begin{align*}
        \MoveEqLeft\sum_{k=1}^d\int_{\R^d} (\partial_k U)^2(\partial_k g)^2\diff\mu = \sum_{k=1}^d\int_{\R^d} (\partial_kg)^2\partial_{kk}U\diff\mu+2\sum_{k=1}^d\int_{\R^d} \partial_k U\partial_{kk}g\partial_kg\diff\mu\\
        &\leq \sum_{k=1}^d\int_{\R^d} (\partial_kg)^2\partial_{kk}U\diff\mu + \sum_{k=1}^d \int_{\R^d}\left( \frac{1}{2}(\partial_kU)^2(\partial_kg)^2 + 2(\partial_{kk}g)^2\right)\diff\mu,
    \end{align*}
    so that rearranging yields
    \begin{equation*}
        \sum_{k=1}^d\int_{\R^d} (\partial_k U)^2(\partial_k g)^2\diff\mu \leq 2\sum_{k=1}^d\int_{\R^d}\partial_{kk}U(\partial_kg)^2\diff\mu + 4\sum_{k=1}^d\int_{\R^d}(\partial_{kk}g)^2\diff\mu. 
    \end{equation*}
    Therefore, by Assumption~\ref{asm:U3},
    \begin{equation}\label{eq:Uggoodbound}
        \sum_{k=1}^d\int_{\R^d} (\partial_k U)^2(\partial_k g)^2\diff\mu\leq 2L\Vert|\nabla g|\Vert_{L^2(\mu)}^2 + 4\big\Vert\Vert\nabla^2g\Vert_F\big\Vert_{L^2(\mu)}^2.  
    \end{equation}
    Hence, using Bochner's identity, \eqref{eq:curvaturebound2} and $\lVert|\nabla g|\rVert_{L^2(\mu)}^2\leq\frac{1}{m}\lVert2\mL g\rVert_{L^2(\mu)}^2$, Assumption \eqref{eq:C1} is satisfied with 
    \begin{equation*}
        C_1^2 = 4+32C_\kappa + \frac{16LC_\kappa+(4+32 C_\kappa)Km}{m}\leq  44(1+K)+20\frac{L}{m}
    \end{equation*}
    in the Gaussian and hypercube case, and with
    \begin{equation*}
        C_1^2 = 18d + \frac{8dL+18dKm}{m} = 18d(1+K) + 8d\frac{L}{m}
    \end{equation*}
    in the case of uniform distribution on the scaled coordinate directions. One then concludes by Remark~\ref{rem:decay_rate_in_simple_setting}.

\end{proof}

If instead of Assumption~\ref{asm:U3} we assume the weaker Assumption~\ref{asm:U2}, we obtain the same rate as for the Forward process in Corollary~\ref{coro:rateforward}, at least for Gaussian or uniform on the hypercube velocities. 
\begin{corollary}
    Let Assumption \ref{asm:U2} be satisfied.
    \begin{enumerate}[(i)]
        \item Let $\kappa = \mathcal{N}(0,I_d)$ or $\kappa = \mathrm{Unif}(\{-1,+1\}^d)$. Then, choosing $T=m^{-1/2}$, the transition semigroup $(\hat P_t)$ of the Zig-Zag process is exponentially contractive in $T$-average with rate
        \begin{equation*}
            \nu = \Omega\left( \frac{\gamma m}{\gamma^2 + \frac{m(1+K)+b}{(1-a)^2}} \right).
        \end{equation*}
         If $\gamma = \Theta\left(\frac{\sqrt{m(1+K)+b}}{1-a}\right)$ then
        $
            \nu = \Omega\left( (1-a)\sqrt{m}\left/\sqrt{(1+K) + b/m} \right.\right).
        $
        \item Let $\kappa = \mathrm{Unif}(\{\pm\sqrt{d}e_i\colon i=1,\dots,n\})$. Then, choosing $T=m^{-1/2}$, the transition semigroup $(\hat P_t)$ of the Zig-Zag process is exponentially contractive in $T$-average with rate
        \begin{equation*}
            \nu = \Omega\left( \frac{\gamma m}{\gamma^2 + \frac{md(1+K)+bd}{(1-a)^2}} \right).
        \end{equation*}
        If $\gamma = \Theta\left(\frac{\sqrt{md(1+K)+bd}}{(1-a)^2}\right)$ then
        $
            \nu = \Omega\left( (1-a)\sqrt{m}\left/\sqrt{((1+K) + b/m)d} \right.\right).
        $
    \end{enumerate}
\end{corollary}
\begin{proof}
    The only difference to the proof of Corollary~\ref{coro:rateZZ1} is the treatment of the term $\sum_{k=1}^d\int_{\R^d}(\partial_kU)^2(\partial_kg)^2\diff\mu$ in the verification of Assumption~\eqref{eq:C1}. We can use \eqref{eq:nablaUnablagbound} to get
    \begin{align*}
        \sum_{k=1}^d\int_{\R^d} (\partial_k U)^2(\partial_k g)^2\diff\mu &\leq \big\Vert|\nabla U||\nabla g|\big\Vert_{L^2(\mu)}^2 \leq \frac{2b(1-a)+4Km}{(1-a)^2}2\E(g) + \frac{4}{(1-a)^2}\Vert 2\mL g\Vert_{L^2(\mu)}^2\\
        &\leq\frac{1}{(1-a)^2}\left(4(1+K)+\frac{2b(1-a)}{m}\right)\lVert 2\mL g\rVert_{L^2(\mu)}^2
    \end{align*}
    Assumption \eqref{eq:C1} then follows with 
    \begin{align*}
        C_1^2 &= 4+4K + \frac{8C_\kappa}{(1-a)^2}\left(4(1+K)+\frac{2b(1-a)}{m}\right)\\&\leq \frac{44(1+K)}{(1-a)^2} + \frac{20b}{(1-a)m}\leq \frac{44(1+K)m+20b}{(1-a)^2m}
    \end{align*}
    in the Gaussian and hypercube case, and with
    \begin{align*}
        C_1^2 &= 2d + 2dK + \frac{4d}{(1-a)^2}\left(4(1+K)+\frac{2b(1-a)}{m}\right) \leq \frac{18d(1+K)m + 8db}{(1-a)^2m}
    \end{align*}
    in the case of uniform distribution on the scaled coordinate directions.
    One then concludes by Remark~\ref{rem:decay_rate_in_simple_setting}.

\end{proof}

\section*{Acknowledgments}

A.~Eberle and F.~Lörler were supported by the Deutsche Forschungsgemeinschaft (DFG, German Research Foundation) under Germany’s Excellence Strategy -- GZ 2047/1, Project-ID 390685813. A.~Eberle and F.~Lörler wurden gefördert durch die Deutsche Forschungsgemeinschaft (DFG) im Rahmen der Exzellenzstrategie des Bundes und der Länder -- GZ 2047/1, Projekt-ID 390685813. 
All the authors acknowledge the support of the French Agence nationale de la recherche under the grant
ANR-20-CE46-0007 (SuSa project). A.~Guillin and M.~Michel  are supported by the ANR-23-CE-40003, Conviviality. A.~Guillin has benefited from a government grant managed by the Agence Nationale de la Recherche under the France 2030
investment plan ANR-23-EXMA-0001.

\bibliography{references}

\end{document}